\documentclass[twoside,centertags]{amsart}

\usepackage{amsfonts}
\usepackage{amsthm}
\usepackage{amsmath}
\usepackage{amssymb}
\usepackage[all]{xy}
\usepackage[latin1]{inputenc}

\theoremstyle{plain}
\newtheorem{theorem}{Theorem}[section]
\newtheorem{corollary}[theorem]{Corollary}
\newtheorem{lemma}[theorem]{Lemma}
\newtheorem{proposition}[theorem]{Proposition}
\theoremstyle{definition}
\newtheorem{defi}[theorem]{Definition}
\theoremstyle{remark}
\newtheorem{obs}[theorem]{Remark}
\newtheorem{ex}[theorem]{Example}

\newcommand{\cat}{\mathcal}
\newcommand{\C}{{\cat C}}

\newcommand{\D}{{\cat D}}
\newcommand{\M}{{\cat M}}

\newcommand{\Nmodel}{{\cat N}}
\newcommand{\calL}{{\mathcal L}}
\newcommand{\calS}{{\mathcal S}}

\newcommand{\DF}{{\mathrm{F}}}

\newcommand{\colim}{\mathop{\text{\rm colim}}}
\newcommand{\hocolim}{\mathop{\text{\rm hocolim}}}

\newcommand{\Ho}{{\text{\rm Ho}}}
\newcommand{\map}{{\text{\rm map}}}
\newcommand{\Map}{{\text{\rm Map}}}
\newcommand{\Hom}{{\text{\rm Hom}}}
\newcommand{\End}{{\text{\rm End}}}

\newcommand{\Ssets}{\text{\rm sSet}}

\newcommand{\Spectra}{\text{\rm Sp}}
\newcommand{\HoSp}{\text{\rm Ho({Sp})}}

\newcommand{\RMod}{\text{$R$-{\rm Mod}}}
\newcommand{\SMod}{\text{$S$-\rm Mod}}
\newcommand{\EMod}{\text{$E$-{\rm Mod}}}
\newcommand{\Gpd}{\text{\rm Gpd}}

\newcommand{\ZZ}{\mathbb{Z}}

\newcommand{\Ass}{A}
\newcommand{\Com}{E}

\newcommand{\FQ}{F}
\newcommand{\LH}{{\mathcal L}}

\SelectTips{cm}{}

\numberwithin{equation}{section}

\begin{document}

\title{Comparing localizations across adjunctions}
\author{Carles Casacuberta}
\address{Institut de Matem\`atica, Facultat de Matem\`atiques i Inform\`atica, Universitat de Barcelona (UB),
Gran Via de les Corts Catalanes 585, 08007 Barcelona, Spain}
\email{carles.casacuberta@ub.edu}
\author{Oriol Ravent\'os}
\address{Fakult\"{a}t f\"{u}r Mathematik, Universit\"{a}t Regensburg,
93040 Regensburg, Germany}
\email{oriol.raventos-morera@ur.de}
\author{Andrew Tonks}
\address{Department of Mathematics, University of Leicester,
Leicester LE1 7RH, United Kingdom}
\email{apt12@leicester.ac.uk}
\date{}

\thanks{The authors were partially supported by the Spanish Ministry of Economy and Competitiveness under grants MTM2010-15831, MTM2013-42178-P, AEI/FEDER grant MTM2016-76453-C2-2-P, and grant MDM-2014-0445 awarded to the Barcelona Graduate School of Mathematics, as well as by Generalitat de Catalunya as members of the 2009~SGR~119, 2014~SGR~114, and 2017~SGR~585 research groups. The second-named author was supported by the project CZ.1.07/2.3.00/20.0003 of the Operational Programme Education for Competitiveness of the Ministry of Education, Youth and Sports of the Czech Republic and by grant SFB 1085 (\emph{Higher Invariants}) of the German Research Foundation.}
\subjclass[2010]{Primary 55P60, 18A40; Secondary 55P48}
\keywords{Homotopy, localization, cellularization, adjunction, monad, comonad, operad}

\begin{abstract}
We show that several apparently unrelated formulas involving left or right Bousfield localizations in homotopy theory are induced 
by comparison maps associated with pairs of adjoint functors. Such comparison maps are used in the article to discuss the existence of functorial liftings of homotopical localizations and cellularizations to categories of algebras over monads acting on model categories, with emphasis on the cases of module spectra and algebras over simplicial operads. Some of our results hold for algebras up to homotopy as well;
for example, if $T$ is the reduced monad associated with a simplicial operad and $f$ is any map of pointed simplicial sets, then $f$\nobreakdash-loc\-al\-iza\-tion coincides with $Tf$-localization on spaces underlying homotopy $T$-algebras, and similarly for cellularizations.
\end{abstract}

\maketitle

\vspace*{-0.3cm}

\section{Introduction}

Preservation of structures such as loop spaces, infinity loop spaces, or module spectra under homotopical localizations or cellularizations has been studied using Segal's theory of loop spaces \cite{Bo94,DF}, operads \cite{CGMV,Javier3,White,WY1,WY2}, algebraic theories \cite{Badzioch,Bergner}, or other methods \cite{Bo96,Bo99,CRT,Javier4}. Monads and their algebras lie behind many of these approaches. However, although the existence of liftings of localizations or colocalizations to categories of algebras over monads has been proved in various special cases, functoriality of such liftings has only been addressed recently in \cite{BW,GRSO,WY3} as well as in the present article.

This article emerged from the observation that some formulas involving localizations or cellularizations in homotopy theory share common patterns not previously revealed. The formulas that we consider ---some of which are well known while others are new--- contain pairs of adjoint functors in some way or another.

As a first example, it was shown by Farjoun in~\cite[Theorem~3.A.1]{DF} that for every pointed connected space $X$ and every basepoint-preserving map $f$ there is a weak equivalence
\begin{equation}
\label{LfOmega}
L_f \Omega X\simeq \Omega L_{\Sigma f} X,
\end{equation}
where $L_f$ denotes localization with respect to $f$ in the homotopy category of pointed spaces, $\Omega$ is the loop functor, and $\Sigma$ denotes suspension. Similarly, Bousfield proved in \cite[Theorem~2.10]{Bo96} that for every spectrum $X$ and every pointed connected space $A$ one has
\begin{equation}
\label{LfOmegainfty}
P_A \Omega^{\infty} X\simeq \Omega^{\infty} P_{\Sigma^{\infty}A}X,
\end{equation}
where $P_A$ is $A$-nullification (that is, localization with respect to the map $A\to *$) and $\Sigma^{\infty}$ is the canonical functor from the homotopy category of pointed spaces to the homotopy category of spectra, while $\Omega^{\infty}$ is its right adjoint.
In Section~\ref{homotopicalstructures} we prove that, as one would expect, the formula \eqref{LfOmegainfty} holds for all $f$-localizations, not only nullifications. 

In a different context, it was shown in \cite[Theorem~1.3]{CRT} that
\begin{equation}
\label{SPinf}
L_f X\simeq L_{SP^{\infty}f} X
\end{equation}
for every basepoint-preserving map $f$ and every commutative topological monoid~$X$, where $SP^{\infty}$ is the infinite symmetric product~\cite{DT}. The stable analogue of this fact is the statement that, if $H\ZZ$ denotes the spectrum that represents homology with integral coefficients, then
\begin{equation}
\label{HZeta}
L_fX\simeq L_{H\ZZ\wedge f}X
\end{equation}
for every $f$ and every spectrum $X$ splitting as a product of Eilenberg--Mac\,Lane spectra. 

Another seemingly unrelated fact was observed in~\cite[Corollary~4.3]{CGT}, namely that there is a natural homomorphism of groupoids
\begin{equation}
\label{piLf}
\pi L_f X\longrightarrow L_{\pi f} (\pi X)
\end{equation}
for all spaces $X$ and every map~$f$, where $\pi$ is the fundamental groupoid. This homomorphism is almost never an isomorphism, yet it becomes an isomorphism after applying $L_{\pi f}$ to~it.

Further details about these examples and many others are given in the article.
Our encompassing approach is based on a study of \emph{comparison maps} of type
\begin{equation}
\label{alphabeta}
\alpha\colon FL_f\longrightarrow L_{Ff}F \qquad \text{and} \qquad
\beta\colon L_fG\longrightarrow GL_{Ff}
\end{equation}
for each Quillen pair of adjoint functors $F$ and $G$ between model categories. In fact $\alpha$ and $\beta$ form a pair of \emph{mates} as in~\cite{Shulman}.
Such comparison maps arise very frequently in practice but are equivalences only in favorable cases. It is remarkable that $\alpha$ becomes an equivalence after applying $L_{Ff}$ to it,
as in~\eqref{piLf}, while $\beta$ does not share this feature.

The formulas \eqref{LfOmega}, \eqref{LfOmegainfty}, \eqref{SPinf} and \eqref{HZeta} are ``of $\beta$ type'' while \eqref{piLf} is ``of $\alpha$ type''. 

Among the new formulas that we obtain by means of comparison maps we emphasize the following ones. For a cofibrant ring spectrum $E$ and a map $f$ of spectra, if $f$-localization commutes with suspension or $E$ is connective, then there is a natural equivalence
\begin{equation}
\label{Elef}
L_fM\simeq L_{E\wedge f}M
\end{equation}
for every left $E$-module $M$. Similarly, for a given spectrum~$A$,
\begin{equation}
\label{CeA}
C_AM\simeq C_{E\wedge A}M
\end{equation}
for all left $E$-modules $M$ if $A$-cellularization commutes with suspension or if $E$ is connective. These results generalize and dualize~\eqref{HZeta}.

The content of the article may be summarized as follows.

\subsection{Comparison maps}

We start with a purely categorical study of natural transformations such as $\alpha$ and $\beta$ in~\eqref{alphabeta}, relating localizations or colocalizations defined in two distinct categories linked by pairs of adjoint functors. This is the core technique of the article, dating back to a preliminary draft written in 2006,
which was gradually adapted for its use in a homotopy-theoretical context. Our task was made feasible by results about transferred model structures on categories of algebras over monads (or algebras over operads) from articles that appeared in the meantime.

\subsection{Preservation of (co)algebras under (co)localizations}

Our main result in Section~\ref{cosection} provides necessary and sufficient conditions under which a localization on a category preserves algebras over a monad defined on the same category. This result uses comparison maps \eqref{alphabeta} associated with the Eilenberg--Moore adjunction of the given monad.

The central finding, inspired by results in~\cite{CRT}, is that, if a monad $T$ and a localization $L$ are defined on the same category, then \emph{$T$ preserves $L$-equivalences if and only if $L$ preserves $T$\nobreakdash-alg\-ebras}. This fact is made precise in Theorem~\ref{monad} and is reminiscent to results in~\cite{Beck}, where distributivity between monads was first discussed.

Theorem~\ref{monad} can be dualized in two different ways. One obvious way consists of passing to opposite categories ---this yields conditions under which colocalizations preserve coalgebras over comonads. The other way is achieved by means of orthogonality between objects and morphisms, leading to conditions under which colocalizations preserve algebras over monads or localizations preserve coalgebras over comonads.

\subsection{Lifting (co)localizations to model categories of algebras}

We subsequently address a homotopy-theoretical version of the results in Section~\ref{cosection}, hence providing necessary and sufficient conditions under which a localization or a colocalization on a model category $\M$ lifts to a category of algebras $\M^T$ over a monad $T$ acting on~$\M$, assuming that the category $\M^T$ admits a model structure transferred from~$\M$; that is, one for which the forgetful functor $U\colon\M^T\to\M$ creates weak equivalences and fibrations. Existence of such transferred model structures (or semi model structures) has been discussed in several articles, such as \cite{BB,BM2, PS,Spitzweck}. Our main results in this part of the article (Theorem~\ref{cofumfumfum} and its dual counterpart, Theorem~\ref{fumfumfum}) were obtained in collaboration with Javier Guti\'errez and David White, who arrived at similar conclusions in \cite{BW,GRSO,WY1}. 

It follows from results in Section~\ref{homotopicalstructures} below or from \cite{BW} that a homotopical $f$-localization lifts to a category of $T$-algebras admitting a transferred model structure if and only if the forgetful functor $U$ sends $Ff$-equivalences to $f$-equivalences, where $F$ is Quillen left adjoint to~$U$, assuming that $L_{Ff}$ exists ---in fact, in this case $L_{Ff}$ is a lift of $L_f$ to $T$-algebras. Similarly, $A$\nobreakdash-cell\-ulariza\-tion lifts to $T$-algebras if and only if the forgetful functor $U$ sends $FA$-cellular objects to $A$-cellular objects, and in this case $C_{FA}$ is a lift of $C_A$ to $T$-algebras, assuming the existence of~$C_{FA}$.

We apply Theorem~\ref{cofumfumfum} and Theorem~\ref{fumfumfum} to modules over a cofibrant ring spectrum $E$ and to algebras over simplicial operads. In both cases, transferred model structures are known to exist. However, there is a crucial distinction: in the case of $E$-module spectra the associated monad $TX=E\wedge X$ preserves all colimits, while in the case of algebras over a simplicial operad $P$ the corresponding monad (whose algebras are the $P$-algebras) only preserves sifted colimits. Therefore our treatment of the case of module spectra is easier.

\subsection{Localizations and cellularizations of module spectra}

We infer that, if $E$ is a cofibrant ring spectrum (in the category of symmetric spectra over simplicial sets), then $f$-localizations and $A$-cellularizations lift to the category of left $E$-modules assuming that they commute with suspension or that $E$ is connective. This was first shown in \cite{CGMV,Javier3} by viewing $E$-module spectra as algebras over a suitable two-colour operad.

Furthermore, if $L_f$ (or $C_A$) commutes with suspension or $E$ is connective, then for every left $E$-module $M$ there are natural equivalences
\begin{align}
\label{formula1}
L_fUM & \simeq L_{U(E\wedge f)}UM\simeq UL_{E\wedge f}M \\[0.1cm]
\label{formula2}
C_{A}UM &\simeq C_{U(E\wedge A)}UM\simeq UC_{E\wedge A}M
\end{align}
that explain and refine \eqref{Elef} and~\eqref{CeA}.

Statements and proofs of claims about cellularizations turn out to be formally analogous to those about localizations, since passage from $L_f$ to $C_A$ involves exchanging objects with morphisms or conversely. However, as our results are formulated in terms of right-induced model structures on categories of algebras over monads, cellularizations tend to behave better than localizations. For example, right Bousfield localizations with respect to objects commute with forming Eilenberg--Moore categories of monads (Theorem~\ref{MATMTFA} below) while left Bousfield localizations with respect to morphisms need not do so (Example~\ref{shocking}).

\subsection{Algebras over simplicial operads}

In Section~\ref{operads} we recover and extend results of Bousfield and Farjoun about interaction of localizations or cellularizations with loop spaces and infinite loop spaces, such as \eqref{LfOmega} and~\eqref{LfOmegainfty}. For this, we prove that 
Theorem~\ref{cofumfumfum} and Theorem~\ref{fumfumfum} apply to model categories of algebras over simplicial operads acting on pointed simplicial sets.
Let us emphasize that our main result in this section (Theorem~\ref{maybe}) not only states that localizations and cellularizations preserve algebras over simplicial operads, but they do it in a functorial way, that is, yielding localization functors and cellularization functors on algebras. In this respect, our approach shares insight with~\cite{GRSO}.

\newpage

\subsection{Algebras up to homotopy}

Examples hinted that some of the above conclusions did not really require the existence of transferred model structures on categories of algebras over monads, but held equally well for algebras up to homotopy, i.e., algebras over the derived monad in the homotopy category. 

For example, the equivalences $L_fM\simeq L_{E\wedge f}M$ and $C_{A}M\simeq C_{E\wedge A}M$ still hold when $M$ is a \emph{homotopy} $E$-module. More generally, we found that
\[
L_fX\simeq L_{Tf}X \qquad \text{and} \qquad C_{A}X\simeq C_{TA}X
\]
if $X$ underlies a homotopy $T$-algebra, provided that $T$ preserves $f$-equivalences and $Tf$\nobreakdash-equiv\-alences, or, correspondingly, $A$-cellular spaces and $TA$-cellular spaces. These assumptions on $T$ are automatically satisfied if $T$ is the reduced monad associated with a simplicial operad acting on pointed simplicial sets. Our proof of this fact relies on arguments used by Farjoun in related parts of~\cite{DF}.

As consequences we infer that
\[
L_f X\simeq L_{\Omega\Sigma f} X \qquad \text{and} \qquad
L_f X\simeq L_{\Omega^{\infty}\Sigma^{\infty}f} X,
\]
assuming in each case that $X$ is a homotopy algebra over the corresponding monad, namely $\Omega\Sigma$ in the first case (whose algebras are the homotopy associative $H$-spaces) and $\Omega^{\infty}\Sigma^{\infty}$ in the second case (whose algebras are called \emph{$H_{\infty}$-spaces}). The second one can be generalized~as
\[
L_f X\simeq L_{\Omega^{\infty}(E\wedge\Sigma^{\infty}f)} X
\]
for algebras over the monad $TX=\Omega^{\infty}(E\wedge\Sigma^{\infty}X)$, where $E$ is any connective ring spectrum. This is an unstable analogue of the equivalence $L_fX\simeq L_{E\wedge f}X$ for homotopy $E$-modules, and if $E=H\ZZ$ then one recovers the formula~\eqref{SPinf}, since $\Omega^{\infty}(H\ZZ\wedge\Sigma^{\infty}X)\simeq SP^{\infty}X$.
 
All these formulas remain valid if $L_f$ is replaced with cellularization $C_A$ with respect to some pointed connected space~$A$.
Most of our results are true, more generally, for localizations with respect to collec\-tions of maps (possibly proper classes) and for colocalizations with respect to collections of objects, whenever these exist. 

Our results in Section~\ref{homotopicalstructures} are conceptually related with the content of \cite[\S\,4]{BCL}, where it is shown that the derived localization of a differential graded algebra $A$ at a set of homogeneous homology classes coincides with the derived localization of $A$ as a left $A$-module. 

Equivalences such as \eqref{formula1} and \eqref{formula2} probably hold, with suitable assumptions on~$E$, for other stable model categories, for example motivic symmetric spectra over any base scheme~$B$, assuming the existence of Bousfield localizations ---this is the case when $B$ is Noetherian and of finite Krull dimension~\cite{DRO}. In this context, \eqref{formula2} may be relevant in the study of the derived category ${\rm DM}(k)$ of motives over a field $k$ of characteristic zero, which is equivalent, according to \cite{RO}, to the homotopy category of modules over the spectrum $M\ZZ$ that represents motivic cohomology for the given field~$k$. 
Note that \eqref{formula2} with $A=S$ (the motivic sphere spectrum) implies that $E$ is $S$-cellular, which is a strong restriction~\cite{DI}. 
Hence, $S$\nobreakdash-cell\-ularity of $E$ is necessary for \eqref{formula2} to hold for all~$A$, perhaps also sufficient.

\vspace*{0.2cm}

\subsection*{Acknowledgements}

We appreciate many conversations with Javier Guti\'er\-rez and David White on the subject of this paper. Some of our results in Section~\ref{homotopicalstructures} have been obtained through exchanges of ideas with them.
The first-named author also benefited from remarks by Boris Chorny, Fernando Muro and Jos\'e Luis Rodr\'{\i}guez, and the second-named author wants to acknowledge useful discussions with Ilias Amrani, John Bourke, George Raptis, and Alexandru Stanculescu. The hospitality of the Max-Planck-Institut f\"ur Mathematik in Bonn and the Institut Mittag-Leffler, where parts of this article were written, is also gratefully acknowledged.

\newpage

\section{Adjunctions, monads and comonads}
\label{Admocomo}

This section contains standard terminology and basic facts that will be used in the article.
More information can be found in \cite[Chapters IV and~VI]{MacLane}.
If $\C$ and $\D$ are categories, we denote by
$F : {\C} \rightleftarrows {\D} : G$
a pair of adjoint functors, with $F$ left adjoint and $G$ right adjoint,
meaning that there are natural bijections of morphism sets
\begin{equation}
\label{adjunction}
{\D}(FX,Y)\cong{\C}(X,GY)
\end{equation}
for $X$ in $\C$ and $Y$ in~$\D$.
We denote by $\varphi^{\rm t}\colon X\to GY$ the adjunct of a morphism
\mbox{$\varphi\colon FX\to Y$} under~\eqref{adjunction} and, similarly,
$\psi^{\rm t}\colon FX\to Y$ denotes the adjunct of \mbox{$\psi\colon X\to GY$}. Adjuncts of identities yield natural transformations
$\eta\colon {\rm Id}\to GF$ (called \emph{unit}) and \mbox{$\varepsilon\colon FG\to {\rm Id}$}
(called \emph{counit}), which, in turn, determine the adjunction
by $\psi^{\rm t}=\varepsilon_Y\circ F\psi$ and $\varphi^{\rm t}=G\varphi\circ\eta_X$.

Passage to opposite categories transforms an adjunction $F : {\C} \rightleftarrows {\D} : G$ into another adjunction $G : {\D}^{\rm op} \rightleftarrows {\C}^{\rm op} : F$, where $G$ is now the left adjoint and the former unit becomes the counit, and conversely,
since $\C(X,GFX)=\C^{\rm op}(GFX,X)$ for all~$X$.

If $F:\C\rightleftarrows \D:G$ is a pair of adjoint functors,
then $F$ is a retract of $FGF$ and $G$ is a retract of~$GFG$.
This follows from the \emph{triangle identities}
\begin{equation}
\label{retraction}
\varepsilon_{FX}\circ F\eta_X=(\eta_X)^{\rm t}={\rm id}_{FX} \quad \text{and} \quad
G\varepsilon_Y\circ\eta_{GY}=(\varepsilon_Y)^{\rm t}={\rm id}_{GY}
\end{equation}
for all objects $X$ and $Y$.
Moreover, if we consider the full subcategories
\[
\C_{\eta} =\{X\in\C\mid\eta_X\colon X\cong GFX\}, \quad
\D_{\varepsilon} =\{Y\in\D\mid\varepsilon_Y\colon FGY\cong Y\},
\]
then $F$ and $G$ restrict to an equivalence of categories
\begin{equation}
\label{Leinster}
F:\C_{\eta}\rightleftarrows \D_{\varepsilon}:G.
\end{equation}

A \emph{monad} on a category $\C$ is a triple $(T,\eta,\mu)$ where $T\colon\C\to\C$ is a functor and $\eta\colon{\rm Id}\to T$ (the \emph{unit}) and $\mu\colon TT\to T$ (the \emph{multiplication}) are natural transformations such that
\[
\mu\circ T\mu=\mu\circ\mu T \quad\text{and}\quad \mu\circ T\eta=\mu\circ\eta T={\rm Id}_T.
\]
A~monad $(T,\eta,\mu)$ is called \emph{idempotent} if $\mu$ is an isomorphism,
which we then omit from the notation.
If $F:{\C}\rightleftarrows{\D}:G$ is an adjunction with unit $\eta$ and counit $\varepsilon$, then $(GF,\eta,G\varepsilon F)$ is a monad. In fact, all monads are of this form, in a non-unique way.

If $(T,\eta,\mu)$ is a monad on a category~$\C$, then a \emph{$T$\nobreakdash-algebra} is a pair $(X,a)$ with $a\colon TX\to X$ such that
$a\circ Ta=a\circ\mu_X$ and
$a\circ\eta_X={\rm id}_X$.
A~morphism of $T$\nobreakdash-algebras $(X,a)\to (Y,b)$
is a morphism $\varphi\colon X\to Y$ in $\C$
such that $\varphi\circ a=b\circ T\varphi$.
Thus, the $T$\nobreakdash-algebras form a category, denoted by
${\C}^T$ and called the \emph{Eilenberg--Moore category} of~$T$, which fits into an adjunction
\begin{equation}
\label{EilenbergMoore}
F:\C\rightleftarrows{\C}^T:U
\end{equation}
where $FX=(TX,\mu_X)$
and $U(X,a)=X$.
This adjunction is terminal among all adjunctions whose associated monad is~$T$.

A \emph{comonad} on a category $\C$ is a monad on the opposite category $\C^{\rm op}$. We denote a comonad by $(T,\varepsilon,\delta)$ where $\varepsilon\colon T\to{\rm Id}$ is the \emph{counit} and $\delta\colon T\to TT$ is the \emph{comultiplication}. A \emph{coalgebra} over $(T,\varepsilon,\delta)$ is a pair $(X,a)$ with $a\colon X\to TX$ such that $Ta\circ a=\delta_X\circ a$ and $\varepsilon_X\circ a={\rm id}_X$. If $F:{\D}\rightleftarrows{\C}:G$ is an adjunction with unit $\eta$ and counit $\varepsilon$, then $(FG,\varepsilon,F\eta G)$ is a comonad. The category $\C_T$ of $T$\nobreakdash-coalgebras provides an initial adjunction
$U:{\C}_T\rightleftarrows\C:G$ among those yielding~$T$, where $GX=(TX,\delta_X)$ for all~$X$; cf.\ \cite[\S\,2]{HS}.

A full subcategory $\calS$ of a category $\C$ is \emph{reflective} if
the inclusion $J$ is part of an adjunction
$K : {\C} \rightleftarrows {\calS} : J$. In this case, the counit $KJ\to {\rm Id}$ is an isomorphism and the functor $L=JK$ is called a \emph{reflector} or a \emph{localization} on~$\C$. If we denote the unit by $l\colon {\rm Id}\to L$, then $(L,l)$ is an idempotent
monad. An object of $\C$ is called \emph{$L$\nobreakdash-local} if it is isomorphic to an object in the subcategory~$\calS$; hence, $X$ is $L$\nobreakdash-local if and only if $l_X\colon X\to LX$ is an isomorphism. A~morphism $g\colon U\to V$ is an \emph{$L$\nobreakdash-equivalence} if $Lg$ is an isomorphism, or, equivalently, if
for all $L$\nobreakdash-local objects $X$ composition with $g$ induces a bijection
\begin{equation}
\label{orthogonality}
{\C}(V,X)\cong{\C}(U,X).
\end{equation}
Conversely, the $L$\nobreakdash-local objects are precisely those $X$ for
which \eqref{orthogonality} holds for all $L$\nobreakdash-equiva\-lences $g\colon U\to V$; see \cite{Adams1} for further details.

Dually, a full subcategory $\calS$ of a category $\D$ is \emph{coreflective} if
the inclusion $J$ has a right adjoint
$J : {\calS} \rightleftarrows {\D} : K$. Then the
functor $C=JK$ is called a \emph{coreflector} or a \emph{colocalization} on~$\D$, and in this case the unit ${\rm Id}\to KJ$ is an isomorphism. If we denote the counit by $c\colon C\to {\rm Id}$, then $(C,c)$ is an idempotent comonad. An object of $\D$ is called \emph{$C$\nobreakdash-colocal} if it is isomorphic to an object in~$\calS$; thus, $X$ is $C$\nobreakdash-colocal if and only if $c_X\colon CX\to X$ is an isomorphism. A~morphism $g\colon U\to V$ is a \emph{$C$\nobreakdash-equivalence} if $Cg$ is an isomorphism, or, equivalently, if for all $C$\nobreakdash-colocal objects $X$ composition with $g$ induces a bijection
\begin{equation}
\label{coorthogonality}
{\D}(X,U)\cong{\D}(X,V).
\end{equation}
The $C$\nobreakdash-colocal objects are precisely those $X$ for which \eqref{coorthogonality} holds for all $C$\nobreakdash-equiv\-alences $g\colon U\to V$. In fact, colocalizations are localizations on the opposite category.

\section{Comparison morphisms}

Suppose given a localization $L_1$ on a category~$\C_1$ and a localization $L_2$ on a category~$\C_2$.
We say that a functor $F\colon {\C}_1\to{\C}_2$ \emph{preserves local objects} if $FX$ is $L_2$\nobreakdash-local for every $L_1$\nobreakdash-local object~$X$, and we say that $F$ \emph{preserves equivalences} if $Ff$ is an $L_2$\nobreakdash-equivalence whenever $f$ is an $L_1$\nobreakdash-equivalence.
We also say that a functor $F\colon{\C}_1\to{\C}_2$ \emph{reflects local objects} if, for an object $X$ of~$\C_1$, the assertion that $FX$ is $L_2$\nobreakdash-local implies that $X$ is $L_1$\nobreakdash-local. Similarly, $F$ \emph{reflects equivalences} if $f$ is an $L_1$\nobreakdash-equivalence whenever $Ff$ is an $L_2$\nobreakdash-equivalence.
The same terminology will be used for colocalizations.

\begin{lemma}
\label{adjprop}
Let $F : {\C}_1 \rightleftarrows {\C}_2 : G$ be a pair of adjoint functors.
\begin{itemize}
\item[(a)]
If localizations on $\C_1$ and $\C_2$ are given, then $G$ preserves local objects if and only if $F$ preserves equivalences.
\item[(b)]
If colocalizations on $\C_1$ and $\C_2$ are given, then $F$ preserves colocal objects if and only if $G$ preserves equivalences.
\end{itemize}
\end{lemma}

\begin{proof}
This follows from the definitions, in view of the commutative diagram
\[
\xymatrix{
\C_2(FB,X)\ar[r] \ar[d]_\cong & \C_2(FA,X) \ar[d]^\cong \\
\C_1(B,GX)\ar[r] & \C_1(A,GX)}
\]
for a morphism $A\to B$ in $\C_1$ and an object $X$ in~$\C_2$,
where the vertical bijections are given by the adjunction. Part~(b) is deduced from part~(a) by passing to the opposite categories.
\end{proof}

\begin{theorem}
\label{alfabeta}
Let $F\colon {\C}_1\to {\C}_2$ be a functor.
Let $L_1$ be a localization on $\C_1$ with unit~$l_1$
and $L_2$ a localization on $\C_2$ with unit~$l_2$.
Then the following hold:
\begin{itemize}
\item[{\rm (i)}]
$F$ preserves equivalences if and only if there is a natural transformation
\[ \alpha\colon FL_1\longrightarrow L_2F \]
such that $\alpha\circ Fl_1=l_2F$.
If this holds, then $\alpha$ is unique and
 $\alpha_X$ is an $L_2$\nobreakdash-equiv\-al\-ence for all~$X$. Moreover, $\alpha$ is an isomorphism if and only if $F$ preserves local objects.
\item[{\rm (ii)}]
$F$ preserves local objects if and only if there is a natural transformation
\[ \beta\colon L_2F\longrightarrow FL_1 \]
such that $\beta\circ l_2F=Fl_1$.
If this holds, then
$\beta$ is unique, and it is an isomorphism if and only if $F$ preserves equivalences.
\end{itemize}
\end{theorem}

\begin{proof}
For every $X$ in ${\C}_2$, the morphism $(l_1)_X\colon X\to L_1X$
is an $L_1$\nobreakdash-equivalence. Therefore, if $F$ preserves equivalences,
then $F(l_1)_X$ is an $L_2$\nobreakdash-equivalence.
Hence it induces a natural bijection
\[
{\C}_2(FL_1X, L_2FX)\cong {\C}_2(FX ,L_2FX)
\]
and $\alpha_X$ is uniquely defined by the equality $\alpha_X\circ F(l_1)_X=(l_2)_{FX}$.
Since $(l_2)_{FX}$ and $F(l_1)_X$ are both $L_2$\nobreakdash-equivalences,
$\alpha_X$ is also an $L_2$\nobreakdash-equivalence.
In order to prove that $\alpha$ is a natural transformation, we need to check
that $\alpha_Y\circ FL_1f$ is equal to $L_2Ff\circ\alpha_X$ for every $f\colon X\to Y$.
This follows from the equality
\[
L_2Ff\circ\alpha_X\circ F(l_1)_X=\alpha_Y\circ FL_1f\circ F(l_1)_X,
\]
using the fact that $F(l_1)_X$ is an $L_2$\nobreakdash-equivalence and $L_2FY$ is $L_2$\nobreakdash-local.

Next we assume that there is a natural transformation $\alpha$ such that $\alpha\circ Fl_1=l_2F$ and infer that $F$ preserves equivalences. For this purpose, let $f\colon X\to Y$ be an $L_1$-equivalence. Since $(l_2)_{FY}$ is an $L_2$-equivalence and $L_2FX$ is $L_2$-local, there is a unique $s\colon L_2FY\to L_2FX$ such that $s\circ (l_2)_{FY}=\alpha_X\circ(FL_1f)^{-1}\circ F(l_1)_Y$. In order to prove that $Ff$ is an $L_2$-equiv\-alence it suffices to check that $s$ is an inverse of $L_2Ff$. This is deduced from the equalities
\[
s\circ L_2Ff\circ(l_2)_{FX}=(l_2)_{FX} \quad \text{and} \quad
L_2Ff\circ s \circ(l_2)_{FY}=(l_2)_{FY}.
\]

If $F$ preserves equivalences, then the equality $\alpha_X\circ F(l_1)_X=(l_2)_{FX}$ implies that $\alpha_X$ is an $L_2$\nobreakdash-equivalence, and therefore $\alpha_X$ is an isomorphism if and only if $FL_1X$ is $L_2$\nobreakdash-local. Hence $\alpha$ is an isomorphism if and only if $F$ sends $L_1$-local objects to $L_2$-local objects. This completes the proof of part~(i).
The proof of part~(ii) is similar.
\end{proof}

\begin{corollary}
\label{LUUL}
Let $F\colon {\C}_1\to {\C}_2$ be a functor.
For a localization $L_1$ on $\C_1$ and a localization $L_2$ on~$\C_2$,
the functors $FL_1$ and $L_2F$ are naturally isomorphic if and only if
$F$ preserves local objects and equivalences. In this case,
$\alpha\colon FL_1\to L_2F$ and $\beta\colon L_2F\to FL_1$ are mutually inverse isomorphisms.
\end{corollary}

\begin{proof}
The ``if'' part follows from Theorem~\ref{alfabeta}.
For the converse, note that if $FL_1\cong L_2F$ then $F$ preserves
local objects, and the naturality of the isomorphism adds the fact
that $F$ preserves equivalences, since, for a morphism $f$,
we have that $L_2Ff$ is an isomorphism if and only if $FL_1f$ is an isomorphism.
Furthermore, if $F$ preserves local objects and equivalences,
then the equality $\alpha\circ\beta\circ l_2F=l_2F$ implies that $\alpha\circ\beta={\rm id}$.
\end{proof}

\begin{corollary}
\label{creating}
Suppose that a functor $U\colon {\C}_1\to {\C}_2$ reflects isomorphisms. For a localization $L_1$ with unit $l_1$ on $\C_1$ and a localization $L_2$ with unit $l_2$ on~$\C_2$, if $UL_1$ and $L_2U$ are naturally isomorphic then the following hold:
\begin{itemize}
\item[{\rm (i)}] $U$ preserves and reflects local objects and equivalences.
\item[{\rm (ii)}] There is a natural isomorphism $\alpha\colon UL_1\to L_2U$ such that $\alpha\circ Ul_1=l_2U$, and a natural isomorphism $\beta\colon L_2U\to UL_1$ such that $\beta\circ l_2U=Ul_1$.
\end{itemize}
Moreover, the isomorphisms in {\rm (ii)} are unique and inverse to each other.
\end{corollary}

\begin{proof}
Corollary~\ref{LUUL} tells us that $U$ preserves local objects and equivalences, and that $\alpha\colon UL_1\to L_2U$ and $\beta\colon L_2U\to UL_1$, as given by Theorem~\ref{alfabeta}, are inverse isomorphisms.
To prove that $U$ reflects local objects, suppose that $UX$ is $L_2$\nobreakdash-local. Then $(l_2)_{UX}$ is an isomorphism. Since $\alpha\circ Ul_1=l_2U$ and $\alpha_X$ is an isomorphism, we infer that $U(l_1)_X$ is an~iso\-morphism and therefore so is $(l_1)_X$ since $U$ reflects isomorphisms. Hence, $X$ is $L_1$\nobreakdash-local, as needed.
The fact that $U$ reflects equivalences follows from the equality
\[
UL_1g\circ\beta_X =\beta_Y\circ L_2Ug
\]
for every $g\colon X\to Y$, together with the fact that $U$ reflects isomorphisms.
\end{proof}

\begin{ex}
For a cocomplete category $\C$ and a small category $I$, choose $F$ to be the colimit functor
$
\colim_I\colon \C^I\to \C,
$
where $\C^I$ denotes the category of functors $I\to\C$.
Let $L$ be a localization on $\C$ and extend it objectwise over $\C^I$; that is, for each $X\colon I\to \C$, define $(LX)_i=L(X_i)$ for all $i\in I$. Then the diagonal functor $\C\to\C^I$ preserves local objects and equivalences. Since it preserves local objects, the colimit functor preserves equivalences by Lemma~\ref{adjprop}. Therefore Theorem~\ref{alfabeta} yields a natural morphism
\begin{equation}
\label{colim}
\alpha\colon \colim_{i\in I} LX_i \longrightarrow L\Big(\colim_{i\in I} X_i\Big)
\end{equation}
which is an $L$-equivalence, that is, it induces an isomorphism
\[
L\Big(\colim_{i\in I} LX_i\Big)\cong L\Big(\colim_{i\in I}X_i\Big).
\]
This is an instance of the well-known fact that left adjoints preserve colimits.

Since the diagonal functor $\C\to\C^I$ also preserves equivalences, if $\C$ is complete then the limit functor preserves local objects. This yields a natural morphism
\[
\beta\colon L\Big(\lim_{i\in I} X_i\Big)\longrightarrow \lim_{i\in I} LX_i
\]
which is rarely an equivalence of any type.
However, for every colocalization $C$ there is an isomorphism
\[
C\Big(\lim_{i\in I}X_i\Big)\cong C\Big(\lim_{i\in I} CX_i\Big)
\]
for each $X\colon I\to \C$, as a special case of the next result, which follows from Theorem~\ref{alfabeta} by passing to the opposite categories.
\end{ex}

\begin{theorem}
\label{coalfabeta}
Let $G\colon {\C_2}\to {\C_1}$ be a functor.
Let $C_1$ be a colocalization on $\C_1$ with counit~$c_1$
and $C_2$ a colocalization on $\C_2$ with counit~$c_2$.
Then the following hold:
\begin{itemize}
\item[{\rm (i)}]
$G$ preserves equivalences if and only if there is a natural transformation
\[ \alpha\colon C_1G\longrightarrow GC_2 \]
such that $Gc_2\circ\alpha =c_1G$.
If this holds, then $\alpha$ is unique and $\alpha_X$ is a $C_1$\nobreakdash-equivalence for all~$X$. Moreover, $\alpha$ is an isomorphism if and only if $G$ preserves colocal objects.
\item[{\rm (ii)}]
$G$ preserves colocal objects if and only if there is a natural transformation
\[ \beta\colon GC_2\longrightarrow C_1G \]
such that $c_1 G\circ\beta=G c_2$.
If this holds, then $\beta$ is unique, and it
is an isomorphism if and only if $G$ preserves equivalences.
\end{itemize}
\end{theorem}

The duals of Corollary~\ref{LUUL} and Corollary~\ref{creating} hold as well.
As a special case of Corollary~\ref{LUUL} and its dual, we obtain conditions for commutativity of a localization and a colocalization.

\begin{corollary}
\label{LC}
If $L$ is a localization and $C$ is a colocalization on the same category, then the following statements are equivalent:
\begin{itemize}
\item[{\rm (a)}]
There is a natural isomorphism $LC\cong CL$.
\item[{\rm (b)}]
$L$ preserves $C$-colocal objects and $C$-equivalences.
\item[{\rm (c)}]
$C$ preserves $L$-local objects and $L$-equivalences.
\end{itemize}
\end{corollary}

\section{Induced (co)localizations on (co)algebras}
\label{cosection}

If often happens that a localization or a colocalization preserves a certain subcategory.
For example, every localization on the category of groups preserves abelian groups \cite{Kortrijk}.
As we next point out, in such cases the restriction is also a localization or a colocalization.
The claim that an endofunctor $F$ on a category $\C$ \emph{preserves} a subcategory $\calS$ means that $FX$ is in $\calS$ for every object $X$ of $\calS$ and $Ff$ is in $\calS$ for every morphism $f$ of~$\calS$. If $\calS$ is full, then the second condition is implied by the first one.

\begin{proposition}
\label{restriction}
Let $\calS$ be a full subcategory of a category $\C$.
\begin{itemize}
\item[(a)]
If a localization $L$ on~$\C$ preserves $\calS$ then $L$ restricts to a localization on~$\calS$, and the inclusion $\calS\hookrightarrow\C$ preserves and reflects local objects and equivalences.
\item[(b)]
If a colocalization $C$ on~$\C$ preserves $\calS$ then $C$ restricts to a colocalization on~$\calS$, and the inclusion $\calS\hookrightarrow\C$ preserves and reflects colocal objects and equivalences.
\end{itemize}
\end{proposition}

\begin{proof}
For (a), consider the full subcategory $\calL$ of $\calS$ consisting of all $L$\nobreakdash-local objects of $\C$ that are in~$\calS$. Then, for each object $X$ in $\calS$, the morphism $l_X\colon X\to LX$
is in $\calS$ by assumption, and, for each $Y$ in~$\calL$, it induces a bijection
\[
\calS(LX,Y)=\C(LX,Y)\cong\C(X,Y)=\calS(X,Y),
\]
so $L$ restricts indeed to a reflection of $\calS$ onto~$\calL$ such that the inclusion
preserves and reflects local objects. If $f\colon X\to Y$ is a morphism in~$\calS$,
then $Lf$ is an isomorphism in $\calS$ if and only if it is an isomorphism in~$\C$,
since the inclusion reflects isomorphisms. Hence, the inclusion also preserves and reflects equivalences. Passage to the opposite category yields~(b).
\end{proof}

In a similar spirit, the following result enhances \cite[Theorem~1.2]{CRT} by adding the remark that, if a localization sends algebras over a monad to algebras over the same monad, then it yields in fact a localization on the category of such algebras. The noteworthy principle contained in this theorem is that, for a monad $T$ and a localization $L$ on the same category, \emph{$T$ preserves $L$-equivalences if and only if $L$ preserves $T$-algebras}.

\begin{theorem}
\label{monad}
Let $(T,\eta,\mu)$ be a monad on a category $\C$ and let $L$ be a localization on $\C$ with unit $l$.
Let $U$ be the forgetful functor from the category $\C^T$ of $T$-algebras to~$\C$. Then the following
statements are equivalent:
\begin{itemize}
\item[{\rm (a)}] $T$ preserves $L$\nobreakdash-equivalences.
\item[{\rm (b)}] For every $T$\nobreakdash-algebra
$(X,a)$ there is a $T$\nobreakdash-algebra structure $\tilde a\colon TLX\to LX$ on $LX$ such that $l_X\colon (X,a)\to (LX,\tilde a)$ is a morphism of $T$\nobreakdash-algebras, and if $g\colon (X,a)\to (Y,b)$ is any morphism of $T$\nobreakdash-algebras then so is~$Lg\colon (LX,\tilde a)\to (LY,\tilde b)$.
\item[{\rm (c)}] There is a localization $L^T$ on $\C^T$ with unit $l^T$ such that $LU=UL^T$ and $lU=Ul^T$.
\item[{\rm (d)}] There is a localization $L^T$ on $\C^T$ together with a natural isomorphism $LU\cong UL^T$.
\end{itemize}

If these statements hold, then the $T$-algebra structures $\tilde a\colon TLX\to LX$ given in part~{\rm (b)} are unique and the localization $L^T$ in~{\rm (c)} is also unique. Moreover, $U$ preserves and reflects local objects and equivalences.
\end{theorem}

\begin{proof}
We first show that (a) $\Rightarrow$ (b).
To obtain $\tilde a\colon TLX\to LX$ with $\tilde a\circ Tl_X=l_X\circ a$, we use the fact that $LX$ is $L$\nobreakdash-local and $Tl_X$ is an $L$\nobreakdash-equivalence by assumption, and note that $\tilde a$ is then unique. The relation $\tilde a\circ T\tilde a=\tilde a \circ\mu_{LX}$ follows, as in the proof of \cite[Theorem~1.2]{CRT}, from the equality
\begin{align*}
\tilde a\circ T\tilde a\circ TTl_X & =
\tilde a\circ Tl_X\circ Ta =
l_X\circ a \circ Ta \\ & =
l_X\circ a\circ \mu_X =
\tilde a\circ Tl_X\circ \mu_X=
\tilde a\circ\mu_{LX}\circ TTl_X,
\end{align*}
since $TTl_X$ is an $L$-equivalence. Similarly, we infer that $\tilde a\circ\eta_{LX}={\rm id}_{LX}$ from the equality $\tilde a\circ\eta_{LX}\circ l_X=l_X$.
Given a morphism of $T$-algebras $g\colon (X,a)\to (Y,b)$, we find that
\[
Lg\circ\tilde a\circ Tl_X = Lg\circ l_X\circ a = l_Y \circ g \circ a = l_Y\circ b\circ Tg =
\tilde b\circ Tl_Y\circ Tg = \tilde b\circ TLg\circ Tl_X,
\]
which implies that $Lg\circ\tilde a=\tilde b\circ TLg$, as claimed.

Next we prove that (b) $\Rightarrow$ (c).
For each $T$\nobreakdash-algebra $(X,a)$, let us define $L^T(X,a)=(LX,\tilde a)$, where $\tilde a$ is given by assumption. Thus,
\[
UL^T(X,a)=LX=LU(X,a)
\]
for every $X$ and all~$a\colon TX\to X$. Moreover, we set $l^T=l$, so we have indeed $lU=Ul^T$.
For a morphism of $T$-algebras $g\colon (X,a)\to (Y,b)$, we define $L^Tg=Lg$, which is a morphism of $T$-algebras by assumption. Hence $UL^Tg=LUg$, and $L^T$ is a functor because so is~$L$.

To check that $L^T$ is a localization, suppose given any morphism $g\colon (X,a)\to (Y,b)$ of $T$\nobreakdash-algebras, and suppose that $Y$ is $L$\nobreakdash-local. Then there is a unique
$g'\colon LX\to Y$ in $\C$ such that $g'\circ l_X=g$, and the equality $l_Y\circ g'\circ l_X=Lg\circ l_X$ implies that $l_Y\circ g'=Lg$ as well. We need to prove that $g'$ is also a morphism of $T$\nobreakdash-algebras, that is, that $g'\circ\tilde a$ is equal to $b\circ Tg'$.
For this, we use the fact that $Lg$ is a morphism of $T$-algebras by assumption to infer that
\[
l_Y\circ g'\circ\tilde a=
Lg\circ\tilde a=\tilde b\circ TLg=\tilde b\circ Tl_Y\circ Tg'=
l_Y\circ b\circ Tg',
\]
and then we use the fact that $l_Y$ is an isomorphism since $Y$ is $L$-local.

The implication (c) $\Rightarrow$ (d) is trivial. Finally, we prove that (d) $\Rightarrow$ (a). Suppose that a natural isomorphism $LU\cong UL^T$ is given. Then Corollary~\ref{creating} tells us that $U$ preserves and reflects local objects and equivalences. Now write $T=UF$ where $FX=(TX,\mu_X)$, and note that $F$ preserves equivalences by Lemma~\ref{adjprop}. Since $U$ also preserves equivalences, so does~$T$, hence yielding~(a). 
\end{proof}

When condition (c) of Theorem~\ref{monad} is satisfied, we say that $L^T$ is a \emph{lifting} of~$L$ to the category $\C^T$ of $T$-algebras. This notion was already considered by Beck in~\cite{Beck}, where it was shown that the existence of a lifting of $L$ to $\C^T$ is in fact equivalent to the existence of a \emph{distributive law of $T$ over~$L$}, that is, a natural transformation $\lambda\colon TL\to LT$ subject to the conditions $\lambda\circ \eta L=L\eta$, $\lambda\circ T l =lT$, and $\lambda\circ \mu L =L\mu \circ \lambda T\circ T \lambda$ (the fourth condition in \cite{Beck} is automatic in our case since $L$ is idempotent). Under these conditions, $LT$ becomes a monad by means of~$\lambda$. Indeed, it follows from part~(i) of Theorem~\ref{alfabeta} that, if $L^T$ is a lifting of $L$ to $\C^T$, then there is a unique natural transformation
\[
\alpha\colon FL\longrightarrow L^TF
\]
under~$F$, where $FX=(TX,\mu_X)$, and thus $\lambda=U\alpha$ is a distributive law of $T$ over~$L$.

The condition that $T$ preserves $L$\nobreakdash-equivalences holds for all localizations $L$ whenever $\C(T-,-)=\Phi\circ\C(-,-)$ for some functor $\Phi$, as in the next example.

\begin{ex}
\label{ex3}
If $\varphi\colon R\to S$ is a central homomorphism of unital rings and $S$ is viewed as an $(R,R)$-bimodule via~$\varphi$, then there is a natural isomorphism
\begin{equation}
\label{rs}
\Hom_R(S\otimes_R M,N)\cong \Hom_R(S,\Hom_R(M,N))
\end{equation}
for any two left $R$-modules $M$ and $N$. Here $TM=S\otimes_R M$ defines a monad on the category $\RMod$ of left $R$-modules, whose algebras are the left $S$-modules. If $L$ is any localization on~$\RMod$, then $T$ preserves $L$-equivalences, since \eqref{rs} yields
\[
\Hom_R(Tf,N)\cong\Hom_R(S,\Hom_R(f,N))
\]
for every $L$-local $R$-module $N$ and every $L$-equivalence $f$ between $R$-modules. Therefore, $L$ lifts to the category $\SMod$ of left $S$-modules. This result generalizes \cite[Theorem~4.3]{CRT}.
\end{ex}

\begin{ex}
\label{AE}
Let $\HoSp$ be the homotopy category of spectra and let $E$ be a homotopy ring spectrum, that is, a monoid in~$\HoSp$. Then $TX=E\wedge X$ defines a monad on~$\HoSp$, whose algebras are left homotopy $E$-module spectra. 

If $A$ is any spectrum, then an \emph{$A_*$-equivalence} is a map of spectra $f\colon X\to Y$ such that $A\wedge f\colon A\wedge X\to A\wedge Y$ is an isomorphism in~$\HoSp$. It was proved in~\cite{BoLoc} that  there is an \emph{$A_*$\nobreakdash-loc\-alization} functor $(-)_A$ on $\HoSp$ whose class of equivalences is precisely the class of $A_*$-equivalences. Since $T$ preserves $A_*$-equivalences, it follows from Theorem~\ref{monad} that $(-)_A$ lifts to the category of $E$-module spectra, as already pointed out in \cite{Bo99,CG}. 
The resulting functor annihilates precisely the $A_*$-acyclic $E$-modules, that is, those $E$-modules $M$ such that $A\wedge M=0$. But the $A_*$-acyclic $E$-modules coincide with the $(E\wedge A)_*$-acyclic ones, since $A\wedge M=0$ implies $E\wedge A\wedge M=0$ and, conversely, if $M$ is an $E$-module with structure map $m\colon E\wedge M\to M$ and we assume that $E\wedge A\wedge M=0$, then the composite
\[
\xymatrix@+0.5cm{A\wedge M \ar[r]^-{A\wedge\eta_M} & A\wedge E\wedge M \ar[r]^-{A\wedge m} & A\wedge M}
\]
is zero and, since it is also the identity map, we conclude that $A\wedge M=0$. This argument yields that
\begin{equation}
\label{MAMEA}
M_A\cong M_{E\wedge A}
\end{equation}
for every left homotopy $E$-module $M$ and every spectrum~$A$. This result was first proved in \cite[Proposition~3.2]{Javier2}, where it was used to obtain a complete description of homological localizations of stable GEMs by choosing $E=H\ZZ$.
\end{ex}

If the monad $(T,\eta,\mu)$ is idempotent, as in the next example, then the conditions of Theorem~\ref{monad} are, in their turn, equivalent to the condition that $L$ preserves the full subcategory of $T$\nobreakdash-local objects, and in this case the lifting $L^T$ is just the restriction of $L$ to this subcategory.

\begin{ex}
If $L$ is any localization on the category of groups, then, as explained in \cite[Theorem~2.2]{Kortrijk},
$L$ preserves abelian groups and hence restricts to the full subcategory of these. Therefore there is a natural group homomorphism
\begin{equation}
\label{abelianization}
\alpha_G\colon (LG)_{\rm ab}\longrightarrow L(G_{\rm ab})
\end{equation}
and a natural isomorphism
$
L((LG)_{\rm ab})\cong L(G_{\rm ab})
$
for all groups $G$ and every localization~$L$. We note, however, that (\ref{abelianization}) is far from being an isomorphism in general. For instance, let $P$ be a set of primes and let $P'$ denote its complement. A~group $G$ is \emph{uniquely $P'$-divisible} if the map $x\mapsto x^q$ is bijective in $G$ for every $q\in P'$. If $l_G\colon G\to G_P$ denotes a universal homomorphism from $G$ into a uniquely $P'$-divisible group, then $(-)_P$ is a localization on the category of groups whose restriction to abelian groups is tensoring with $\ZZ_P$. As shown in~\cite{Baumslag}, if $F$ is a free group of rank~$n\ge 2$ then $(F_P)_{\rm ab}\cong (\ZZ_P)^n\oplus T$ where $T$ is a nonzero $P'$\nobreakdash-torsion group. Thus the group $(F_P)_{\rm ab}$ is not uniquely $P'$-divisible, although $((F_P)_{\rm ab})_P\cong (F_{\rm ab})_P$.
\end{ex}

As next shown, Theorem~\ref{monad} also holds for coalgebras over a comonad. The proof is totally analogous, but the roles of local objects and equivalences are exchanged. Thus, the motto is now that \emph{a comonad $T$ preserves $L$\nobreakdash-local objects if and only if $L$ preserves $T$\nobreakdash-coalgebras}.

\begin{theorem}
\label{comonad}
Let $(T,\varepsilon,\delta)$ be a comonad on a category $\C$ and let $L$ be a localization on $\C$ with unit $l$. Let $U$ be the forgetful functor from the category $\C_T$ of $T$-coalgebras to~$\C$. Then the following statements are equivalent:
\begin{itemize}
\item[{\rm (a)}] $T$ preserves $L$\nobreakdash-local objects.
\item[{\rm (b)}] For every $T$\nobreakdash-coalgebra
$(X,a)$ there is a $T$\nobreakdash-coalgebra structure $\tilde a\colon LX\to TLX$ on $LX$ such that $l_X\colon (X,a)\to (LX,\tilde a)$ is a morphism of $T$\nobreakdash-coalgebras, and if $g\colon (X,a)\to (Y,b)$ is any morphism of $T$\nobreakdash-coalgebras then so is $Lg\colon (LX,\tilde a)\to (LY,\tilde b)$.
\item[{\rm (c)}] There is a localization $L_T$ on $\C_T$ with unit $l_T$ such that $LU=UL_T$ and $lU=Ul_T$.
\item[{\rm (d)}] There is a localization $L_T$ on $\C_T$ together with a natural isomorphism $LU\cong UL_T$.
\end{itemize}

If these statements hold, then the $T$-coalgebra structures $\tilde a\colon LX\to TLX$ given in part~{\rm (b)} are unique, and the localization $L_T$ in~{\rm (c)} is also unique. Moreover, $U$ preserves and reflects local objects and equivalences.
\end{theorem}

\begin{proof}
Suppose that $T$ preserves $L$-local objects.
In order to obtain $\tilde a\colon LX\to TLX$ with $\tilde a\circ l_X=Tl_X\circ a$, use that $TLX$ is $L$\nobreakdash-local by assumption and $l_X$ is an $L$\nobreakdash-equiv\-al\-ence, and note that $\tilde a$ is then unique. The relations $T\tilde a\circ\tilde a=\delta_{LX}\circ\tilde a$ and $\varepsilon_{LX}\circ\tilde a={\rm id}_{LX}$ are consequences of the equalities $T\tilde a\circ\tilde a\circ l_X=\delta_{LX}\circ\tilde a\circ l_X$ and $\varepsilon_{LX}\circ\tilde a\circ l_X=l_X$.
The localization $L_T$ is defined by $L_T(X,a)=(LX,\tilde a)$ and the rest of the proof follows the same steps as the proof of Theorem~\ref{monad}. The implication (d)~$\Rightarrow$~(a) follows from Lemma~\ref{adjprop}, as the right adjoint $FX=(TX,\varepsilon_X)$ preserves local objects if and only if $U$ preserves equivalences.
\end{proof}

By passing to opposite categories, Theorem~\ref{comonad} yields a result about preservation of algebras over a monad under the effect of a colocalization, which we next state for later reference, and Theorem~\ref{monad} dualizes into a result relating colocalizations with coalgebras over a comonad, which we omit.
Distributivity between comonads was previously studied in~\cite{Barr}. Distributive laws of monads over comonads have been considered in \cite[\S\,2.3]{HKRS}.

\begin{theorem}
\label{comonaddual}
Let $(T,\eta,\mu)$ be a monad on a category $\C$ and let $C$ be a colocalization on $\C$ with counit $c$. Let $U$ be the forgetful functor from the category $\C^T$ of $T$-algebras to~$\C$. Then the following statements are equivalent:
\begin{itemize}
\item[{\rm (a)}] $T$ preserves $C$\nobreakdash-colocal objects.
\item[{\rm (b)}] For every $T$\nobreakdash-algebra
$(X,a)$ there is a $T$\nobreakdash-algebra structure $\tilde a\colon TCX\to CX$ on $CX$ such that $c_X\colon (CX,\tilde a)\to (X,a)$ is a morphism of $T$\nobreakdash-algebras, and if $g\colon (X,a)\to (Y,b)$ is any morphism of $T$\nobreakdash-algebras then so is $Cg\colon (CX,\tilde a)\to (CY,\tilde b)$.
\item[{\rm (c)}] There is a colocalization $C^T$ on $\C^T$ with unit $c^T$ such that $CU=UC^T$ and \hbox{$cU=Uc^T$}.
\item[{\rm (d)}] There is a colocalization $C^T$ on $\C^T$ together with a natural isomorphism $CU\cong UC^T$.
\end{itemize}

If these statements hold, then the $T$-algebra structures $\tilde a\colon TCX\to CX$ given in part~{\rm (b)} are unique, and the colocalization $C^T$ in {\rm (c)} is also unique. Moreover, $U$ preserves and reflects colocal objects and equivalences.
\end{theorem}

\begin{proof}
The given monad $T$ defines a comonad on $\C^{\rm op}$ whose coalgebras are the $T$-algebras, and $C$ defines a localization on $\C^{\rm op}$ whose local objects are the $C$-colocal ones. Thus all the statements follow from those in Theorem~\ref{comonad}.
\end{proof}

\begin{ex}
As in Example~\ref{ex3}, if $\varphi\colon R\to S$ is a central ring homomorphism then we may consider the monad defined by $TM=S\otimes_R M$ on the category $\RMod$ of left $R$-modules. If $C$ is any colocalization on~$\RMod$, then $T$ preserves $C$-colocal objects, since \eqref{rs} yields
\begin{equation}
\label{trick}
\Hom_R(TN,f)\cong\Hom_R(S,\Hom_R(N,f))
\end{equation}
for every $C$-colocal $R$-module $N$ and every $C$-equivalence $f$ between $R$-modules. Therefore, $C$ lifts to the category $\SMod$ of left $S$-modules, as first shown in \cite[Proposition~2.1]{GRS}.
\end{ex}

\begin{ex}
If $C$ is any colocalization on the category of groups, then, according to \cite{DFGS}, $C$~preserves nilpotent groups of any nilpotency class $k$ and hence restricts to the full subcategory of these. This yields a natural group homomorphism
\begin{equation}
\label{nilpotent}
\beta_G\colon CG/\Gamma_k CG \longrightarrow C(G/\Gamma_k G)
\end{equation}
for every group $G$ and all~$k$, where $\Gamma_k$ denotes the $k$th term of the lower central series. For $k=1$, the homomorphism \eqref{nilpotent} takes the form
$(CG)_{\rm ab}\to C(G_{\rm ab})$, which is not an isomorphism in general, not even a $C$\nobreakdash-equivalence. For example, the $3$-torsion subgroup of the symmetric group $\Sigma_3$ is cyclic of order~$3$ while the abelianization of $\Sigma_3$ has no $3$-torsion.
\end{ex}

\section{Inverting morphisms and building from objects}
\label{invertingonemorphism}

A collection of morphisms ${\mathcal F}$ and a collection of objects ${\mathcal X}$ in a category~$\C$ are called \emph{orthogonal} if for every morphism $f\colon A\to B$ in ${\mathcal F}$ and every object $X$ in ${\mathcal X}$ the induced function
\begin{equation}
\label{orthogonal}
{\C}(f,X)\colon {\C}(B,X)\longrightarrow {\C}(A,X)
\end{equation}
is a bijection of sets. The objects orthogonal to ${\mathcal F}$ are called \emph{${\mathcal F}$\nobreakdash-local} and the morphisms orthogonal to the collection of all ${\mathcal F}$\nobreakdash-local objects are called \emph{${\mathcal F}$\nobreakdash-equivalences}.

An \emph{${\mathcal F}$\nobreakdash-localization} of an object $X$ is an ${\mathcal F}$\nobreakdash-equivalence $l_X\colon X\to L_{\mathcal F}X$ into an ${\mathcal F}$\nobreakdash-local object.
If an ${\mathcal F}$\nobreakdash-localization exists for all objects, then $L_{\mathcal F}$ is indeed a localization on~$\C$.
Note that, if a localization functor $L$ is given, then $L=L_{\mathcal F}$ where ${\mathcal F}$ is the collection of all $L$\nobreakdash-equivalences.

As shown in \cite[Theorem~1.39]{AR}, if a category $\C$ is locally presentable, then
$L_{\mathcal F}$ exists whenever ${\mathcal F}$ is a set, while if $\mathcal F$ is a proper class then the existence of $L_{\mathcal F}$ can be proved if one assumes the existence of sufficiently large cardinals \cite[Chapter~6]{AR}.

Dually, a collection of morphisms ${\mathcal F}$ and a collection of objects ${\mathcal A}$ in a category~$\C$ are \emph{co-orthogonal} if for every morphism $f\colon X\to Y$ in ${\mathcal F}$ and every object $A$ in ${\mathcal A}$ the induced function
\begin{equation}
\label{coorthogonal}
{\C}(A,f)\colon {\C}(A,X)\longrightarrow {\C}(A,Y)
\end{equation}
is a bijection of sets. The morphisms co-orthogonal to ${\mathcal A}$ are called \emph{${\mathcal A}$\nobreakdash-equivalences} and the objects co-orthogonal to the collection of all ${\mathcal A}$\nobreakdash-equivalences are called \emph{${\mathcal A}$\nobreakdash-colocal}. 

An \emph{${\mathcal A}$\nobreakdash-colocalization} of an object $X$ is
an ${\mathcal A}$\nobreakdash-equivalence $c_X\colon C_{\mathcal A} X\to X$ from an ${\mathcal A}$\nobreakdash-colocal object into~$X$.
If an ${\mathcal A}$\nobreakdash-colocalization exists for every object, then $C_{\mathcal A}$ is a colocalization on~$\C$. If a colocalization $C$ is given, then $C=C_{\mathcal A}$ where $\mathcal A$ is the collection of all $C$\nobreakdash-colocal objects.

If $\C$ is a locally presentable category, then the existence of $C_{\mathcal A}$ is ensured for every set ${\mathcal A}$ by the dual of the Special Adjoint Functor Theorem \cite[\S\,0.7]{AR}, since the full subcategory of $\mathcal A$\nobreakdash-colocal objects is closed under colimits and has a generating set (namely $\mathcal A$), and $\C$ is cowellpowered according to \cite[Theorem~1.58]{AR}. If $\mathcal A$ is a proper class, then the existence of $\C_{\mathcal A}$ also follows from convenient large-cardinal axioms \cite[Theorem~6.28]{AR}.

Since the main source of motivation of the present article is the study of localizations of the form $L_f$ for a single morphism~$f$ and colocalizations of the form $C_A$ for a single object~$A$, we will restrict statements of results to those cases, also for the sake of simplicity. However, most of our conclusions hold equally well for collections of morphisms instead of a single morphism, and for collections of objects instead of a single object, provided that the corresponding localizations or colocalizations exist.

\begin{proposition}
\label{FGf}
Let $F: \C_1\rightleftarrows \C_2 :G$ be a pair of adjoint functors and let
$f\colon A\to B$ be a morphism in~$\C_1$.
\begin{itemize}
\item[{\rm (i)}]
An object $Y$ in $\C_2$ is $Ff$\nobreakdash-local if and only if $GY$ is $f$\nobreakdash-local.
\item[{\rm (ii)}]
$F$ sends $f$-equivalences to $Ff$-equivalences.
\item[{\rm (iii)}]
If $L_f$ and $L_{Ff}$ exist, then there is a unique natural transformation
\[
\alpha\colon F{L_f}\longrightarrow L_{Ff}F
\]
such that $\alpha_X\circ Fl_X=l_{FX}$ for all~$X$,
and $\alpha_X$ is an $Ff$-equivalence for all~$X$. There is also a unique natural transformation
\[
\beta\colon L_f G\longrightarrow G L_{Ff}
\]
such that $\beta_Y\circ l_{GY}=Gl_Y$ for all~$Y$.
Furthermore, $\alpha$ is an isomorphism if and only if $F$ preserves local objects, and $\beta$ is an isomorphism if and only if $G$ preserves equivalences.
\end{itemize}
\end{proposition}

\begin{proof}
For every object $Y$ in~$\C_2$, consider the commutative diagram
\[\xymatrix{
\C_2(FB,Y)\ar[r]
\ar[d]_\cong & \C_2(FA,Y) \ar[d]^\cong \\
\C_1(B,GY)\ar[r]
& \C_1(A,GY)\rlap{,} }
\]
where the vertical bijections are given by the adjunction and the horizontal arrows are induced by $Ff$ and~$f$. It follows that $GY$ is $f$\nobreakdash-local if and only if $Y$ is
$Ff$\nobreakdash-local, as claimed. Part~(ii) is a consequence of~(i),
and all the claims in~(iii) follow from Theorem~\ref{alfabeta}.
\end{proof}

The natural transformations $\alpha$ and $\beta$ in part~(iii) of Proposition~\ref{FGf} are \emph{mates}; that is, each of them determines the other one as follows
(see \cite[p.\,5]{Shulman}):
\[
\beta = GL_f\varepsilon \circ G\alpha G\circ \eta L_fG, \qquad
\alpha = \varepsilon L_{Ff} F \circ F\beta F \circ FL_f\eta,
\]
where $\eta$ is the unit of the adjunction and $\varepsilon$ is the counit.

The analogue of Proposition~\ref{FGf} for colocalizations reads as follows.

\begin{proposition}
\label{coFGf}
Let $F: \C_1\rightleftarrows \C_2 :G$ be a pair of adjoint functors and $A$ an object in~$\C_1$.
\begin{itemize}
\item[{\rm (i)}]
A morphism $g$ in $\C_2$ is an $FA$\nobreakdash-equivalence if and only if $Gg$ is an $A$\nobreakdash-equiv\-al\-ence.
\item[{\rm (ii)}]
$F$ sends $A$-colocal objects to $FA$-colocal objects.
\item[{\rm (iii)}]
If $C_A$ and $C_{FA}$ exist, then there is a unique natural transformation
\[
\alpha\colon C_AG \longrightarrow G{C_{FA}}
\]
such that $Gc_Y\circ\alpha_Y=c_{GY}$ for all~$Y$, and $\alpha_Y$ is an $A$-equivalence for all~$Y$.
There is also a unique natural transformation
\[
\beta\colon F C_A\longrightarrow C_{FA} F
\]
such that $c_{FX}\circ\beta_X=Fc_X$ for all $X$.
Furthermore, $\alpha$ is an isomorphism if and only if $G$ preserves colocal objects, and $\beta$ is an isomorphism if and only if $F$ preserves equivalences.
\end{itemize}
\end{proposition}

\begin{proof}
Parts (i) and (ii) are proved using the commutative diagram
\[\xymatrix{
\C_2(FA,X)\ar[r]
\ar[d]_\cong & \C_2(FA,Y) \ar[d]^\cong \\
\C_1(A,GX)\ar[r]
& \C_1(A,GY)}
\]
for every $g\colon X\to Y$ in $\C_2$,
where the vertical bijections are given by the adjunction and the horizontal arrows are induced by $g$ and $Gg$. Part (iii) comes from Theorem~\ref{coalfabeta}.
\end{proof}

\begin{proposition}
\label{greattheorem}
Let $F:\C\rightleftarrows\C^T:U$ be the Eilenberg--Moore factorization
of a monad $T$ on a category~$\C$.
Let $f$ be a morphism in $\C$ such that $L_f$ exists.
Then the following statements are equivalent:
\begin{itemize}
\item[{\rm (i)}]
$T$ preserves $f$\nobreakdash-equivalences.
\item[{\rm (ii)}]
$L_{Ff}$ exists and there is a natural isomorphism $L_fU\cong UL_{Ff}$.
\end{itemize}

Moreover, if $T$ preserves both $f$\nobreakdash-equivalences and $Tf$-equivalences and $L_{Tf}$ exists, then there is a natural isomorphism
$L_fU\cong L_{Tf}U$.
\end{proposition}

\begin{proof}
Suppose first that $T$ preserves $f$\nobreakdash-equivalences. Theorem~\ref{monad}
tells us that there is a localization $L^T$ on $\C^T$ equipped with a natural isomorphism $L_fU\cong UL^T$, for which $U$ preserves and reflects local objects and equivalences. Hence the $L^T$\nobreakdash-local objects are those $(X,a)$ in $\C^T$ such that $X$ is $f$\nobreakdash-local. But part~(i) of Proposition~\ref{FGf} tells us that $X$ is $f$\nobreakdash-local if and only if $(X,a)$ is $Ff$\nobreakdash-local. Hence $L^T$ is indeed an $Ff$\nobreakdash-localization (which therefore exists).
Conversely, a natural isomorphism $L_fU\cong UL_{Ff}$ implies that $T$ preserves $f$\nobreakdash-equivalences, according to Theorem~\ref{monad}. This proves that (i)~$\Leftrightarrow$~(ii).

Now assume that $T$ preserves $f$\nobreakdash-equivalences and $Tf$-equivalences, and that $L_{Tf}$ exists. Then the implication (i) $\Rightarrow$ (ii) yields natural isomorphisms
\begin{equation}
\label{here}
L_fU\cong UL_{Ff} \quad \text{and} \quad L_{Tf}U\cong UL_{FTf}.
\end{equation}

Since $T$ preserves $f$\nobreakdash-equivalences, $Tf$ is an $f$\nobreakdash-equivalence. Hence all $f$\nobreakdash-local objects are $Tf$\nobreakdash-local. If a $T$-algebra $(X,a)$ is $Ff$-local, then $X$ is $f$-local by part~(i) of Proposition~\ref{FGf}, and hence $X$ is $Tf$-local. But this implies that $(X,a)$ is $FTf$-local, again by part~(i) of Proposition~\ref{FGf}. Conversely, since $T=UF$, the fact that $F$ is a retract of $FUF$ by \eqref{retraction} implies that $Ff$ is an $FTf$\nobreakdash-equivalence, and consequently every $FTf$\nobreakdash-local $T$-algebra is $Ff$\nobreakdash-local. This allows us to conclude that the classes of $Ff$-local $T$-algebras and $FTf$-local $T$-algebras coincide, so $L_{Ff}\cong L_{FTf}$.
Therefore $L_fU\cong L_{Tf}U$ by~\eqref{here}, as claimed.
\end{proof}

If $T$ is an idempotent monad on a category~$\C$ and we denote
by $J\colon\calS\to\C$ the inclusion of the full subcategory of $T$\nobreakdash-local objects and by $K\colon\C\to\calS$ its left adjoint, then, as a special case of Proposition~\ref{greattheorem}, we infer, for a morphism $f$ in~$\C$, the following facts:
\begin{itemize}
\item[{\rm (i)}]
If $L_f$ preserves~$\calS$, then
there is a natural isomorphism $L_fJ\cong JL_{Kf}$.
\item[{\rm (ii)}]
If $L_{Tf}$ also preserves~$\calS$,
there is a natural isomorphism $L_fJ\cong L_{Tf}J$.
\end{itemize}
In fact, the proof is easier, since the counit $\varepsilon$ of the adjunction
is in this case an isomorphism and hence $K\cong KJK$.

\begin{ex}
As an example, let $T$ be abelianization on the category of groups. Since all localizations on groups preserve abelian groups, part~(ii) tells us that
\[
L_fA\cong L_{f_{\rm ab}}A
\]
for every group homomorphism $f$ and all abelian groups~$A$.
This fact was used in~\cite{CRT}.
\end{ex}

As in Section~\ref{cosection}, there is an analogous version of Proposition~\ref{greattheorem} for colocalizations:

\begin{proposition}
\label{cogreattheorem}
Let $F:\C\rightleftarrows\C^T:U$ be the Eilenberg--Moore factorization of a monad $T$ on a category~$\C$. Let $A$ be an object in $\C$ such that $C_A$ exists.
Then the following statements are equivalent:
\begin{itemize}
\item[{\rm (i)}]
$T$ preserves $A$-colocal objects.
\item[{\rm (ii)}]
$C_{FA}$ exists and there is a natural isomorphism $C_AU\cong UC_{FA}$.
\end{itemize}

Moreover, if $T$ preserves both $A$\nobreakdash-colocal objects and $TA$\nobreakdash-colocal objects and $C_{TA}$ exists, then there is a natural isomorphism
$C_AU\cong C_{TA}U$.
\end{proposition}

\begin{proof}
The assumption that $T$ preserves $A$-colocal objects implies, by Theorem~\ref{comonaddual}, that $C_A$ induces a colocalization $C^T$ on $\C^T$ such that $C_AU\cong UC^T$ naturally, and $U$ preserves and reflects local objects and equivalences.
Hence the $C^T$\nobreakdash-colocal
objects are those $(X,a)$ in $\C^T$ such that $X$ is $A$\nobreakdash-colocal. From
Proposition~\ref{coFGf} we then infer that $X$ is $A$\nobreakdash-colocal if and only if $(X,a)$ is $FA$\nobreakdash-colocal. Hence $C^T$ is an $FA$\nobreakdash-colocalization, as claimed. The remaining steps are analogous to those in the proof of Proposition~\ref{greattheorem}.
\end{proof}

\begin{ex}
As observed in Example~\ref{ex3},
if $\C(T-,-)$ depends functorially on $\C(-,-)$, then $T$ preserves
$f$\nobreakdash-equivalences for \emph{every}~$f$, and in such cases the assumptions that $T$
preserves $f$\nobreakdash-equivalences and $Tf$\nobreakdash-equiv\-al\-en\-ces in Proposition~\ref{greattheorem} are automatically fulfilled. This happens, for instance, if $\varphi\colon R\to S$ is a central ring homomorphism and $TA=S\otimes_R A$, where $\C$ is the category of left $R$-modules.
Thus we infer from Proposition~\ref{greattheorem} that there is a natural isomorphism
\begin{equation}
\label{Sotimesf}
L_fM\cong L_{S\otimes_R f}M
\end{equation}
for every left $S$\nobreakdash-module $M$ and every homomorphism $f$ of left $R$-modules.
Proposition~\ref{greattheorem} also tells us that there is no ambiguity in the right-hand term of (\ref{Sotimesf}), as it may indistinctly mean the underlying $R$-module of the localization of $M$ with respect to $S\otimes_R f$ in the category of left $S$\nobreakdash-modules or the localization of the $R$-module underlying $M$ with respect to the $R$-module homomorphism underlying $S\otimes_R f$, that is,
\[
L_fUM\cong UL_{S\otimes_R f}M \cong L_{U(S\otimes_R f)}UM.
\]

Similarly, $T$ preserves $A$-colocal objects for every left $R$-module $A$, and therefore we infer from Proposition~\ref{cogreattheorem} that there is a natural isomorphism
\begin{equation}
\label{coSotimesf}
C_AM\cong C_{S\otimes_R A}M
\end{equation}
for every left $S$\nobreakdash-module $M$ and every left $R$-module $A$.
In fact, \eqref{coSotimesf} really means that
\[
C_AUM\cong UC_{S\otimes_R A}M \cong C_{U(S\otimes_R A)}UM.
\]
\end{ex}

Both Proposition~\ref{greattheorem} and Proposition~\ref{cogreattheorem} have duals.
If $L_f$ is a localization on~$\C$ for some morphism~$f$, then $L_f$ may be viewed as a colocalization on the opposite category $\C^{\rm op}$ whose colocal objects are those that are $f$-local in $\C$, that is, those objects that are orthogonal to $f$ in $\C$ and hence co-orthogonal to $f$ in $\C^{\rm op}$. Although it is actually the same functor~$L_f$, this colocalization could be denoted by~$C_f$ for consistency. 
Similarly, if $C_A$ is a colocalization on~$\C$ for some object~$A$, then it is a localization on~$\C^{\rm op}$ whose equivalences are those morphisms that are $A$\nobreakdash-equiv\-alences in~$\C$, and we denote this localization by~$L_A$. 
Accordingly, we call \emph{$A$\nobreakdash-local} those objects in a category that are $A$-colocal in the opposite category. In other words, the class of $A$-local objects in a category is the closure of $A$ under double passage to the orthogonal complement. This notation and terminology is not of common use, so we confine it to the statement of the next result.
The assumption that $L_{A}$ exists is really restrictive, since the opposite of a locally presentable category is not locally presentable. 

\begin{corollary}
\label{cocogreattheorem}
Let $U:\C_T\rightleftarrows\C:F$ be the Eilenberg--Moore factorization of a comonad $T$ on a category~$\C$. Let $A$ be an object in $\C$ such that $L_{A}$ exists.
Then the following statements are equivalent:
\begin{itemize}
\item[{\rm (i)}]
$T$ preserves $A$-local objects.
\item[{\rm (ii)}]
$L_{FA}$ exists and there is a natural isomorphism $L_{A}U\cong UL_{FA}$.
\end{itemize}

Moreover, if $T$ preserves both $A$\nobreakdash-local objects and $TA$\nobreakdash-local objects and $L_{TA}$ exists, then there is a natural isomorphism
$L_A U\cong L_{TA}U$.
\end{corollary}

\section{Homotopical localizations and cellularizations}

In the remaining sections we discuss localizations and colocalizations
in a homotopical context, using the formalism of Quillen model categories~\cite{Quillen}, which we assume equipped with functorial factorizations.
Every model category $\M$ can be endowed with \emph{homotopy function complexes} as described in \cite[\S\,17.5]{Hirschhorn} or \cite[Chapter~5]{Hovey}. We denote by $\map_{\M}(-,-)$ any choice of such which is functorial in both variables.

An adjunction between model categories
$F: \M \rightleftarrows \Nmodel : G$
is a \emph{Quillen adjunction} if $F$ preserves cofibrations and trivial cofibrations or, equivalently, if $G$ preserves fibrations and trivial fibrations.
If this is the case, then they give rise to a \emph{derived} adjunction
\begin{equation}
\label{derivedadjunction}
FQ: \Ho(\M) \rightleftarrows \Ho(\Nmodel) : GR
\end{equation}
between the corresponding homotopy categories,
after having chosen a cofibrant replacement functor $Q$ on $\M$ and a fibrant replacement functor $R$ on~$\Nmodel$.
It then follows, as explained in \cite[Proposition~17.4.16]{Hirschhorn},
that the induced map
\begin{equation}
\label{enrichedadjunction}
\map_{\Nmodel}(FQ X,Y)\longrightarrow\map_{\M}(X,GR Y)
\end{equation}
is a natural weak equivalence of simplicial sets
for $X$ in $\M$ and $Y$ in $\Nmodel$, whose value at $\pi_0$
coincides with the bijection given by the derived adjunction~(\ref{derivedadjunction}).

Unless otherwise specified, in the next sections we will consider this kind of simplicially enriched orthogonality between objects and maps in model categories. Thus, an object $X$ and a map $f\colon V\to W$ in a model category $\M$ will be called \emph{simplicially orthogonal} (or, as in \cite[\S\,17.8]{Hirschhorn}, \emph{homotopy orthogonal})~if
\begin{equation}
\label{enrichedorthogonal}
\map_{\M}(f,X)\colon \map_{\M}(W,X)\longrightarrow \map_{\M}(V,X)
\end{equation}
is a weak equivalence of simplicial sets.

The objects that are simplicially orthogonal to every map in a given collection $\mathcal F$ are called \emph{$\mathcal F$\nobreakdash-local}, and the maps that are simplicially orthogonal to the collection of all $\mathcal F$\nobreakdash-local objects are called \emph{$\mathcal F$\nobreakdash-equivalences}. For convenience we omit in this article the standard assumption that $\mathcal F$-local objects have to be fibrant (thus, our convention is that an object weakly equivalent to an $\mathcal F$-local object is $\mathcal F$-local). Every $\mathcal F$-equivalence between $\mathcal F$-local objects is a weak equivalence. 

If $\M$ satisfies suitable assumptions ---for instance, if it is cofibrantly generated and left proper, and its underlying category is locally presentable, as for pointed or unpointed simplicial sets, Bousfield--Friedlander spectra \cite{BF}, symmetric spectra over simplicial sets \cite{HSS}, groupoids \cite{CGT}, and many other cases---
then for every set of maps $\mathcal F$ there exists a model category $\M_{\mathcal F}$, called \emph{left Bousfield localization} of $\M$ with respect to~$\mathcal F$, with the same underlying category as $\M$ and the same cofibrations, in which the weak equivalences are the $\mathcal F$\nobreakdash-equivalences and the fibrant objects are those that are fibrant in $\M$ and $\mathcal F$\nobreakdash-local; see \cite[\S 3.3]{Hirschhorn} for details. Left Bousfield localizations still exist if $\mathcal F$ is a proper class, provided that a suitable large-cardinal axiom holds~\cite{RT}.

An \emph{$\mathcal F$\nobreakdash-localization} of an object $X$ of $\M$ is
a trivial cofibration $l_X\colon X\to L_{\mathcal F}X$ in $\M_{\mathcal F}$ with $L_{\mathcal F}X$ fibrant in~$\M_{\mathcal F}$. Thus if $\M_{\mathcal F}$ exists then $L_{\mathcal F}$ is a fibrant replacement functor on~$\M_{\mathcal F}$. As such, $(L_{\mathcal F},l)$ is a monad on $\M$ that is idempotent on the homotopy category~$\Ho(\M)$. Therefore the classes of $\mathcal F$\nobreakdash-equivalences and $\mathcal F$\nobreakdash-local objects determine each other by ordinary orthogonality in~$\Ho(\M)$.

In what follows, a \emph{homotopical localization} on a model category $\M$ will mean an $\mathcal F$\nobreakdash-local\-ization $L_{\mathcal F}$ for some collection of maps~$\mathcal F$.
This terminology is consistent with previous articles such as \cite{Chorny, CGMV}.
Although our results hold indeed for arbitrary collections of maps, we will state them for a single map $f$ for simplicity of notation, since most motivating cases involve one map only.

The next result is a homotopical version of Proposition~\ref{FGf}.
As explained in \cite[\S\,3.1.11]{Hirschhorn}, we need to impose suitable fibrancy and cofibrancy assumptions due to the fact that $F$ preserves weak equivalences between cofibrant objects but not all weak equivalences in general, and $G$ preserves weak equivalences between fibrant objects.

\begin{proposition}
\label{HoFGf}
Let $F: \M \rightleftarrows \Nmodel :G$ be a Quillen adjunction between model categories.
For a map $f\colon V\to W$ in~$\M$ between cofibrant objects, the following assertions hold:
\begin{itemize}
\item[{\rm (i)}]
A fibrant object $Y$ in $\Nmodel$ is $\FQ f$\nobreakdash-local if and only if $G Y$ is $f$\nobreakdash-local.
\item[{\rm (ii)}]
$\FQ$ sends $f$\nobreakdash-equivalences between cofibrant objects to $\FQ f$\nobreakdash-equivalences.
\item[{\rm (iii)}]
If the left Bousfield localizations $\M_f$ and $\Nmodel_{\FQ f}$ exist, then for every cofibrant object $X$ in $\M$ there is a homotopy unique and homotopy natural $\FQ f$-equivalence
\[
\alpha_X\colon \FQ{L_f}X \longrightarrow L_{\FQ f} \FQ X
\]
such that $\alpha_X\circ \FQ l_X= l_{\FQ X}$, and for every object $Y$ in $\Nmodel$ there is a homotopy unique and homotopy natural map
\[
\beta_Y\colon L_f GY \longrightarrow GL_{\FQ f}Y
\]
such that $\beta_Y\circ l_{GY}= Gl_Y$.
Moreover, $\alpha_X$ is a weak equivalence if and only if $\FQ L_fX$ is $\FQ f$-local, and $\beta_Y$ is a weak equivalence if and only if $Gl_Y$ is an $f$\nobreakdash-equivalence.
\end{itemize}

If $F$ preserves all weak equivalences then the cofibrancy assumptions are not necessary, and if $G$ preserves all weak equivalences then the fibrancy assumption in part~{\rm (i)} can be omitted.
\end{proposition}

\begin{proof}
For a fibrant object $Y$ in ${\Nmodel}$, consider the commutative diagram
\[\xymatrix
{
\map_{\Nmodel}(\FQ W,Y)\ar[r] 
\ar[d]_\simeq & \map_{\Nmodel}(\FQ V,Y) \ar[d]^\simeq
\\
\map_{\M}(W,G Y)\ar[r]
& \map_{\M}(V,G Y) , }
\]
where the vertical weak equivalences are given by \eqref{enrichedadjunction} since $F$ preserves cofibrant objects and $G$ preserves fibrant ones, and the horizontal arrows are induced by $Ff$ and~$f$. It follows that, if $Y$ is fibrant, then $GY$ is $f$\nobreakdash-local if and only if $Y$ is $\FQ f$\nobreakdash-local, as claimed in part~(i).

Therefore, $G$ sends fibrant $\FQ f$\nobreakdash-local objects to $f$\nobreakdash-local objects, and this implies that $\FQ$ sends $f$\nobreakdash-equivalences between cofibrant objects to $\FQ f$\nobreakdash-equivalences.
This proves~(ii).

Now (ii) implies that, for every cofibrant $X$ in~$\M$, the map $\FQ l_X$ is an $\FQ f$\nobreakdash-equiv\-al\-ence and hence a trivial cofibration in~$\Nmodel_{Ff}$.
Since $L_{\FQ f}\FQ X$ is fibrant in~$\Nmodel_{Ff}$, there is a map
\[
\alpha_X\colon \FQ L_fX\longrightarrow L_{\FQ f}\FQ X
\]
such that $\alpha_X\circ \FQ l_X= l_{\FQ X}$, and $\alpha_X$ is unique up to homotopy with this property. It also follows that $\alpha_X$ is an $\FQ f$\nobreakdash-equivalence since $\FQ l_X$ and $l_{\FQ X}$ are $\FQ f$\nobreakdash-equivalences, and $\alpha_X$ is a weak equivalence if and only if $\FQ L_fX$ is $\FQ f$\nobreakdash-local (although it need not be fibrant). Naturality of $\alpha$ after passing to the homotopy categories is proved as in Theorem~\ref{alfabeta}.

Since the map $l_{GY}$ is a trivial cofibration in the model category $\M_f$ and $GL_{\FQ f}Y$ is fibrant in $\M_f$ by part~(i), there is a map
\[
\beta_Y\colon L_fGY\longrightarrow GL_{\FQ f}Y
\]
such that $\beta_Y\circ l_{GY}=Gl_Y$. It is unique up to homotopy since $GL_{\FQ f}Y$ is $f$-local and $l_{GY}$ is an $f$-equivalence, and it is natural up to homotopy for similar reasons.
Moreover, $\beta_Y$ is a weak equivalence if and only if it is an $f$-equivalence, which happens if and only if $Gl_Y$ is an $f$-equivalence.

If $F$ preserves weak equivalences then for every object $X$ we may choose a cofibrant approximation $\tilde X\to X$ and we have $F\tilde X\simeq FX$, so the previous arguments hold with $\tilde X$ in the place of $X$. Similarly, if $G$ preserves weak equivalences, then for each $Y$ we may choose a fibrant approximation $Y\to\hat Y$ and use $\hat Y$ instead of~$Y$.
\end{proof}

\begin{ex}
The standard model structure on the category $\Gpd$ of groupoids \cite{CGT} has a
simplicial enrichment given by $N{\rm Hom}(-,-)$, where ${\rm Hom}(G,H)$
denotes the groupoid of functors $G\to H$ and $N$ denotes the nerve. Since all groupoids are fibrant and cofibrant, $N{\rm Hom}(G,H)$ is a homotopy function complex from $G$ to $H$ in~$\Gpd$.

Fundamental groupoid and nerve form a Quillen adjunction
\[
\pi : \Ssets\rightleftarrows \Gpd : N,
\]
where $\Ssets$ denotes the category of simplicial sets,
and Proposition~\ref{HoFGf} yields a morphism
\begin{equation}
\label{cgt}
\alpha_X\colon \pi L_f X\longrightarrow L_{\pi f}(\pi X)
\end{equation}
for all $X$ and all~$f$, which is the natural $\pi f$\nobreakdash-equivalence studied in~\cite{CGT}. From the fact that the morphism \eqref{cgt} is a $\pi f$-equivalence it follows that, if $X$ is $1$-connected, then $L_{\pi f}(\pi L_fX)$ is trivial. It is a long-standing open problem to decide if $L_fX$ is in fact $1$-connected for every map $f$ when $X$ is $1$-connected.

It should be possible to extend (\ref{cgt}) to higher dimensions by using
suitable categories of algebraic models for $n$\nobreakdash-types, in which
homotopical localizations can be effectively computed,
as done in \cite{CGT} for the model category of groupoids.
This might yield relevant information
on the $n$\nobreakdash-type of~$L_fX$, which is usually difficult to relate
with the $n$\nobreakdash-type of~$X$.

Let us however emphasize that $L_{\pi f}(\pi X)$ is very different from the space $L_{P_1 f}(P_1 X)$, where $P_1$ denotes the first Postnikov section, that is, $P_1X=K(\pi_1(X,x_0),1)$. The spaces
$P_1L_fX$ and $L_{P_1f}(P_1X)$ are not $P_1f$-equivalent in general. For example, let $f$ be a map between wedges of circles such that $L_fX$
is the localization of $X$ at the prime~$3$; thus $P_1f=f$.
If $X=K(\Sigma_3,1)$ where $\Sigma_3$ is the symmetric group on three letters, then, as shown in \cite[Example~8.2]{CP}, the space $L_fX=K(\Sigma_3,1)_{(3)}$
is $1$-connected since $\Sigma_3$ is generated by elements of order~$2$,
yet it is not contractible since $\Sigma_3$ has nonzero mod~$3$ homology.
Therefore, $P_1L_fX$ is contractible while $L_{P_1f}(P_1X)=L_fX$ is not.
Although $P_1$ is left adjoint to the inclusion
of the full subcategory of $1$\nobreakdash-coconnected spaces, the functor
$L_{P_1f}$ does not restrict to this subcategory; that is, $L_{P_1f}(P_1X)$ need not be a $K(G,1)$.
\end{ex}

If $\mathcal A$ is any collection of objects in a model category~$\M$, then a map $g\colon X\to Y$ will be called an \emph{$\mathcal A$\nobreakdash-equiv\-alence} if
\[
\map_{\M}(A,g)\colon \map_{\M}(A,X)\longrightarrow\map_{\M}(A,Y)
\]
is a weak equivalence of simplicial sets for every $A\in{\mathcal A}$. Correspondingly, an object $B$ (for the purposes of this article, not necessarily cofibrant) is called \emph{$\mathcal A$\nobreakdash-co\-local} (or, more commonly, \emph{$\mathcal A$\nobreakdash-cellular}) if $\map_{\M}(B,g)$ is a weak equivalence for every $\mathcal A$\nobreakdash-equivalence $g$. 

As shown in \cite[\S\,5.1]{Hirschhorn}, if suitable assumptions are imposed on $\M$
then there exists a model category $\M_{\mathcal A}$ for every set of objects $\mathcal A$, called \emph{right Bousfield localization} of $\M$ with respect to~$\mathcal A$, with the same underlying category as $\M$ and the same fibrations, in which the weak equivalences are the $\mathcal A$-equivalences and the cofibrant objects are those that are cofibrant in $\M$ and $\mathcal A$\nobreakdash-colocal.
In order to ensure the existence of $\M_{\mathcal A}$ it is sufficient to assume that $\M$ is right proper and cofibrantly generated~\cite[\S\,5]{Barwick}. In fact the latter condition can be weakened, as done in~\cite[\S\,2]{CI}.

An \emph{$\mathcal A$\nobreakdash-co\-loc\-alization} (or \emph{$\mathcal A$\nobreakdash-cell\-ular\-ization}) of an object $X$ of $\M$ is a trivial fibration $c_X\colon C_{\mathcal A}X\to X$ in $\M_{\mathcal A}$ with $C_{\mathcal A}X$ cofibrant in~$\M_{\mathcal A}$. 
Thus if $\M_{\mathcal A}$ exists then $C_{\mathcal A}$ is a cofibrant replacement functor on~$\M_{\mathcal A}$. As such, $(C_{\mathcal A},c)$ is a comonad on $\M$ that is idempotent on the homotopy category~$\Ho(\M)$.

A \emph{homotopical colocalization} on a model category $\M$ will mean an $\mathcal A$\nobreakdash-cellular\-ization $C_{\mathcal A}$ for some collection of objects~$\mathcal A$.
We will state our results for a single object $A$ in the rest of the article for consistency with the examples, although most of our conclusions remain valid for collections of objects.

\newpage

\begin{proposition}
\label{coHoFGf}
Let $F: \M \rightleftarrows \Nmodel :G$ be a Quillen adjunction between model categories. For a cofibrant object $A$ in~$\M$, the following assertions hold:
\begin{itemize}
\item[{\rm (i)}]
A map $g$ in $\Nmodel$ between fibrant objects is an $FA$\nobreakdash-equivalence if and only if $Gg$ is an $A$\nobreakdash-equiv\-al\-ence.
\item[{\rm (ii)}]
$F$ sends cofibrant $A$\nobreakdash-cellular objects to $FA$\nobreakdash-cellular objects.
\item[{\rm (iii)}]
If the right Bousfield localizations $\M_A$ and $\Nmodel_{FA}$ exist, then for every fibrant object $Y$ in $\Nmodel$ there is a homotopy unique and homotopy natural $A$\nobreakdash-equivalence
\[
\alpha_Y\colon C_{A}G Y \longrightarrow G{C_{FA}} Y
\]
such that $Gc_Y\circ\alpha_Y=c_{GY}$, and for every object $X$ in $\M$ there is a homotopy unique and homotopy natural map
\[
\beta_X\colon F C_A X\longrightarrow C_{FA} F X
\]
such that $c_{FX}\circ\beta_X=Fc_X$. Furthermore, $\alpha_Y$ is a weak equivalence if and only if $GC_{FA}Y$ is $A$\nobreakdash-cellular,
and $\beta_X$ is a weak equivalence if and only if $Fc_X$ is an $FA$\nobreakdash-equivalence.
\end{itemize}

If $F$ preserves all weak equivalences then the cofibrancy assumptions are not necessary, and if $G$ preserves all weak equivalences then the fibrancy assumptions can be omitted.
\end{proposition}

\begin{proof}
This is proved in the same way as Proposition~\ref{HoFGf}.
\end{proof}

\begin{obs}
\label{cellular}
If an object $C$ is $B$-cellular and $B$ is $A$-cellular, then $C$ is $A$-cellular. To check this, let $g\colon X\to Y$ be any $A$-equivalence. Then $\map_{\M}(B,g)$ is a weak equivalence. Hence $g$ is a $B$-equivalence, and this implies that $\map_{\M}(C,g)$ is a weak equivalence, as needed.

A similar argument in the case of localizations shows that if a map $h$ is a $g$-equivalence and $g$ is an $f$-equivalence, then $h$ is an $f$-equivalence.
\end{obs}

\begin{proposition}
\label{FAFTAFfFTf}
Let $F: \M \rightleftarrows \Nmodel :G$ be a Quillen adjunction between model categories and let $T=GF$ be the associated monad.
\begin{itemize}
\item[{\rm (a)}]
Let $A$ be a cofibrant object of $\M$. If $TA$ is cofibrant and $A$\nobreakdash-cell\-ular, then the classes of $FA$-cellular objects and $FTA$-cell\-ular objects coincide.
\item[{\rm (b)}]
Let $f$ be a map in $\M$ between cofibrant objects.
If $Tf$ is an $f$\nobreakdash-equivalence between cofibrant objects, then the classes of $Ff$-equivalences and $FTf$\nobreakdash-equiv\-alences coincide.
\end{itemize}
\end{proposition}

\begin{proof}
In part (a), since $TA$ is cofibrant and $A$\nobreakdash-cellular by assumption, we infer from part~(ii) of Proposition~\ref{coHoFGf} that $FTA$ is $FA$\nobreakdash-cellular. Conversely, $FA$ is $FTA$\nobreakdash-cellular since $F$ is a retract of~$FUF=FT$ by~\eqref{retraction}. Hence, by Remark~\eqref{cellular}, the classes of $FA$-cellular objects and $FTA$-cellular objects are equal.

Part (b) is proved with the same argument, using part~(ii) of Proposition~\ref{HoFGf}.
\end{proof}

We conclude this section with a result relating certain localizations with cellularizations, followed by examples that will be relevant in Section~\ref{operads}. In a pointed model category, localization with respect to a map $A\to *$ is called \emph{$A$-nullification} and denoted by~$P_A$. A~motivating example is $A=S^{n+1}$ for $n\ge 0$ in the category of pointed simplicial sets, for which $P_A$ is the $n$th Postnikov section. In this example, $C_A$ is the $n$-connected cover, so there is a homotopy fibre sequence
\begin{equation}
\label{CAPA}
C_AX\longrightarrow X\longrightarrow P_AX
\end{equation}
for every space~$X$ if $A=S^{n+1}$. An analogous sequence exists for spectra. In fact, \eqref{CAPA} is a homotopy fibre sequence of spectra if and only if a certain condition stated in \cite[Theorem~3.6]{Javier4} holds, as is the case if $A$ is any suspension of the sphere spectrum, and also whenever $C_A$ (and hence also $P_A$) commutes with suspension. The extent to which \eqref{CAPA} fails to be a homotopy fibre sequence for spaces in general was discussed in~\cite{Chacholski}.

\begin{theorem}
\label{LfCA}
Let $A$ be a cofibrant object in a pointed model category $\M$ such that $C_A$ and $P_A$ exist and for every object $X$ the natural sequence
\[
C_AX\longrightarrow X\longrightarrow P_AX
\]
is a homotopy fibre sequence. Let $f\colon V\to W$ be a map between $A$-cellular objects such that $L_f$ exists. Then there is a natural equivalence $L_fC_A\simeq C_AL_f$.
\end{theorem}

\begin{proof}
Since $V$ and $W$ are $A$-cellular, if $Z$ is $f$-local then there are weak equivalences
\[
\map_{\M}(W,C_AZ)\simeq\map_{\M}(W,Z)\simeq\map_{\M}(V,Z)\simeq\map_{\M}(V,C_AZ)
\]
showing that $C_AZ$ is also $f$-local. Thus $C_A$ preserves $f$-local objects, so our result will follow by applying Corollary~\ref{LC} within $\Ho(\M)$ if we prove that $C_A$ preserves $f$-equivalences. 

With this purpose, let $g\colon X\to Y$ be an $f$-equivalence, and consider the commutative diagram
\begin{equation}
\label{CAPAXY}
\xymatrix{C_AX \ar[r] \ar[d]_{C_A(g)} & X \ar[r] \ar[d]^g & P_AX \ar[d]^{P_A(g)} \\
C_AY \ar[r] & Y \ar[r] & P_AY\rlap{.}}
\end{equation}
Since $P_AV\simeq *$ and $P_AW\simeq *$, the map $f$ is a $P_A$-equivalence. Therefore, by Remark~\ref{cellular}, every $f$-equivalence is a $P_A$-equivalence, so in particular $P_A(g)$ is a weak equivalence. Now, if $Z$ is any $f$-local object, then \eqref{CAPAXY} yields by \cite[Corollary~6.4.2(c)]{Hovey} a commutative diagram whose rows are homotopy fibre sequences of simplicial sets:
\begin{equation}
\label{mapsonCAPAXY}
\xymatrix{\map_{\M}(P_AY,Z) \ar[r] \ar[d]_{\simeq} & \map_{\M}(Y,Z) \ar[r] \ar[d]^{\simeq} & \map_{\M}(C_AY,Z) \ar[d] \\
\map_{\M}(P_AX,Z) \ar[r] & \map_{\M}(X,Z) \ar[r] & \map_{\M}(C_AX,Z)\rlap{.}}
\end{equation}
This diagram shows that $C_A(g)\colon C_AX\to C_AY$ is an $f$-equivalence, as needed.
\end{proof}

\begin{ex}
\label{connective}
In the category of pointed simplicial sets, if $f$ is a map between connected spaces, then a space $X$ is $f$-local if and only if its basepoint component $X_0$ is $f$-local. Since $X_0\simeq C_{A}X$ for $A=S^1$, Theorem~\ref{LfCA} yields a natural equivalence $L_fX_0\simeq (L_fX)_0$ for all~$X$.

More generally, if $f$ is any map between $n$-connected spaces for $n\ge 0$, and $X\langle n\rangle$ denotes the $n$-connected cover of a space~$X$, then $X$ is $f$-local if and only if $X\langle n\rangle$ is $f$-local, so Theorem~\ref{alfabeta} yields a natural map
\[
\beta\colon L_fX\langle n\rangle\longrightarrow (L_fX)\langle n\rangle.
\]
which is weak equivalence by Theorem~\ref{LfCA}. This fact was first found in \cite[Theorem~5.2]{CRT}. It does no longer hold if $f$ is not a map between $n$-connected spaces; for example, it is well-known that localization with respect to $K$-theory lowers connectivity~\cite{Mislin}.

Theorem~\ref{LfCA} also implies that, if $f$ is any map between $n$-connected spectra, then $L_fX\langle n\rangle\simeq (L_fX)
\langle n\rangle$ for every spectrum $X$ and all~$n\in\ZZ$. As a special case, 
\begin{equation}
\label{loccoloc}
L_{\Sigma^{\infty}f}X^c\simeq (L_{\Sigma^{\infty}f}X)^c
\end{equation}
for every map $f$ of spaces and every spectrum $X$, where the superscript denotes the connective cover. By \emph{connective} we mean $(-1)$-connected, that is, $X^c\simeq C_SX$, where $S$ is the sphere spectrum.
\end{ex}

\section{Preservation of algebras over monads in model categories}
\label{homotopicalstructures}

In what follows we consider monads on model categories arising from Quillen adjunctions.
For a monad $T$ acting on a model category $\M$, we denote by $\M^T$ the category of $T$-algebras, as in the previous sections. 

A model structure on $\M^T$ is called \emph{transferred} or \emph{right-induced} from the model structure on $\M$ if the forgetful functor $U$ creates (that is, preserves and reflects) weak equivalences and fibrations in the Eilenberg--Moore factorization of $T$,
\begin{equation}
\label{FU}
F:\M\rightleftarrows\M^T:U.
\end{equation}
Transferred model structures on categories of algebras over monads were considered by Batanin--Berger \cite{BB}, Batanin--White \cite{BW}, Elmendorf--Kriz--Mandell--May \cite[Chapter~VII, \S\,4]{EKMM}, Johnson--Noel \cite[\S\,3]{JN}, Schwede--Shipley \cite{SS}, and others.
The list of references would be much longer if the extensive work made on transferred model structures since \cite[Theorem~3.3]{Crans} was to be summarized. Sufficient conditions can be found in \cite[Theorem~11.3.2]{Hirschhorn} for a transferred model structure to exist assuming that $\M$ is cofibrantly generated, and specific conditions for the case of categories of algebras over monads were given in \cite[Theorem~2.11]{BB} and in \cite[Lemma~2.3]{SS}.

An instance where a transferred model structure on algebras over a monad does \emph{not} exist will be described in Example~\ref{shocking} below.

If $\M^T$ admits a transferred model structure, then the forgetful functor $U$ is right Quillen and hence \eqref{FU} is a Quillen adjunction.
Therefore $T$ preserves weak equivalences between cofibrant objects, since $T=UF$ and $F$ preserves weak equivalences between cofibrant objects while $U$ preserves all weak equivalences.
Moreover, there is a derived adjunction
\begin{equation}
\label{factorT}
FQ:\Ho(\M)\rightleftarrows\Ho(\M^T):U,
\end{equation}
where $Q$ is a cofibrant replacement functor on $\M$, and fibrant replacement on $\M^T$ is omitted since $U$ preserves all weak equivalences. 

The adjunction \eqref{factorT} is \emph{not} equivalent, in general, to the Eilenberg--Moore adjunction of the derived monad $UFQ=TQ$ on $\Ho(\M)$.
From now on, for simplicity, we will omit the cofibrant replacement functor $Q$ from the notation, thus writing $\Ho(\M)^{T}$ instead of $\Ho(\M)^{TQ}$.
Counterexamples showing that the categories $\Ho(\M^T)$ and $\Ho(\M)^{T}$ need not be equivalent can be found in \cite{Javier1} and~\cite{JN}. 
In order to emphasize the distinction, we will call \emph{$T$-algebras up to homotopy} the objects of~$\Ho(\M)^{T}$, while we will keep calling $T$-algebras those of~$\Ho(\M^T)$.

The question of when an $f$-localization or an $A$-cellularization on $\M$ induces respectively a localization or a colocalization on $T$-algebras up to homotopy can be quickly answered using results from Section~\ref{cosection}, as follows.

\begin{theorem}
\label{easy}
For a monad $T$ acting on a model category $\M$, suppose that the category~$\M^T$ of $T$-algebras has a transferred model structure.
\begin{itemize}
\item[{\rm (a)}]
Let $f$ be a map in $\M$ such that $L_f$ exists. Then 
$L_f$ lifts to $\Ho(\M)^{T}$ if and only if $T$ sends $f$-equivalences between cofibrant objects to $f$-equivalences.
\item[{\rm (b)}]
Let $A$ be an object in $\M$ such that $C_A$ exists. Then 
$C_A$ lifts to $\Ho(\M)^{T}$ if and only if $T$ sends cofibrant $A$-cellular objects to $A$-cellular objects.
\end{itemize}
\end{theorem}

\begin{proof}
Let $Q$ be a cofibrant replacement functor on~$\M$.
Theorem~\ref{monad} tells us that $L_f$ lifts to the Eilenberg--Moore category $\Ho(\M)^{T}$ if and only if $TQ$ preserves $f$-equivalences in $\Ho(\M)$, and this happens if and only if $T$ sends  $f$\nobreakdash-equivalences between cofibrant objects in~$\M$ to $f$-equivalences. Indeed, if $g\colon X\to Y$ is an $f$-equivalence then so is the map $Qg\colon QX\to QY$, and if $T$ sends $f$-equivalences between cofibrant objects to $f$-equivalences then $TQg$ is an $f$-equivalence. Conversely, suppose that $TQ$ preserves $f$-equivalences. If $g\colon X\to Y$ is an $f$-equivalence where $X$ and $Y$ are cofibrant, then $TQg$ is an $f$-equivalence by assumption. Moreover $TQX\simeq TX$ and $TQY\simeq TY$ since $T$ preserves weak equivalences between cofibrant objects, so $Tg$ is an $f$-equivalence.

In the case of $A$-cellularizations we use Theorem~\ref{comonaddual} to infer that $C_A$ lifts to $\Ho(\M)^{T}$ if and only if $TQ$ preserves $A$-cellular objects in $\Ho(\M)$. The assertion (b) follows since $TQX\simeq TX$ for every cofibrant object~$X$.
\end{proof}

We next address the existence of liftings of $f$-localizations to the homotopy category of $T$-algebras $\Ho(\M^T)$. This lifting problem and the corresponding one for cellularizations have also been discussed in \cite{BW,White,WY1,WY2,WY3} and in~\cite{Javier3,Javier4,GRSO,GV}.

We start by formalizing the notion that a homotopical localization or a homotopical colocalization on $\M$ lift to $\Ho(\M^T)$.

\begin{defi}
\label{preserving}
Let $T$ be a monad on a model category $\M$ such that the category $\M^T$ of $T$-algebras has a transferred model structure. 
\begin{itemize}
\item[(a)]
We say that a homotopical localization $(L,l)$ on $\M$ \emph{lifts to $\Ho(\M^T)$}
if there is an endofunctor $L^T$ on $\Ho(\M^T)$ equipped with a natural transformation $l^T\colon {\rm Id}\to L^T$ and a natural isomorphism $h\colon LU\to UL^T$ such that $h\circ lU = Ul^T$ in~$\Ho(\M)$.
\item[(b)]
We say that a homotopical colocalization $(C,c)$ on $\M$ \emph{lifts to $\Ho(\M^T)$}
if there is an endofunctor $C^T$ on $\Ho(\M^T)$ equipped with a natural transformation $c^T\colon C^T\to {\rm Id}$ and a natural isomorphism $h\colon CU\to UC^T$ such that $Uc^T\circ h = cU$ in~$\Ho(\M)$.
\end{itemize}
\end{defi}

As next shown, it follows automatically that $(L^T,l^T)$ is then a localization and, correspondingly, $(C^T,c^T)$ is a colocalization.

\begin{lemma}
\label{elaprima}
If a homotopical localization $(L,l)$ lifts to $\Ho(\M^T)$ then $(L^T,l^T)$ is a localization on $\Ho(\M^T)$ and the forgetful functor $U$ preserves and reflects local objects and equivalences.
Similarly, if a homotopical colocalization $(C,c)$ lifts to $\Ho(\M^T)$ then $(C^T,c^T)$ is a colocalization on $\Ho(\M^T)$ and the forgetful functor $U$ preserves and reflects colocal objects and equivalences.
\end{lemma}

\begin{proof}
In order to prove that $(L^T,l^T)$ is an idempotent monad, it suffices to check
that $l^TL^T$ and $L^Tl^T$ are isomorphisms on $\Ho(\M^T)$. On one hand,
$h L^T\circ lUL^T = Ul^TL^T$. Moreover, $Lh\circ lLU = lUL^T\circ h$ and $lL$ is an isomorphism in $\Ho(\M)$, so $Ul^TL^T$ is also an isomorphism. Since $U$ reflects isomorphisms, $l^TL^T$ is an isomorphism.
On the other hand, by the naturality of~$h$, we have
$h L^T\circ LUl^T = UL^Tl^T\circ h$. Here $LUl^T = Lh\circ LlU$, and, since $h$ and $Ll$ are isomorphisms, it follows that $LUl^T$ is an isomorphism and hence so is $UL^Tl^T$. Again, since $U$ reflects isomorphisms, we conclude that $L^Tl^T$ is an isomorphism.

The fact that $U$ preserves and reflects local objects and equivalences follows from the isomorphism $LU\cong UL^T$ as detailed in Corollary~\ref{creating},
and the argument for colocalizations is analogous.
\end{proof}

In what follows we treat first the case of colocalizations for simplicity. For an object $A$ in a model category $\M$, we indistinctly denote by $C_A$ the $A$-cellularization functor on $\M$ and the induced colocalization on $\Ho(\M)$, although it is important not to confuse them.

\begin{theorem}
\label{cofumfumfum}
Let $T=UF$ be the Eilenberg--Moore factorization of a monad $T$ on a model category $\M$ such that the category $\M^T$ of $T$-algebras has a transferred model structure. Let $A$ be a cofibrant object in $\M$ such that the right Bousfield localizations $\M_A$ and $(\M^T)_{FA}$ exist. Then the following statements are equivalent:
\begin{itemize}
\item[{\rm (i)}]
$C_A$ lifts to $\Ho(\M^T)$.
\item[{\rm (ii)}]
$U$ sends $FA$-cellular objects to $A$-cellular objects.
\item[{\rm (iii)}]
The comparison map $\alpha_Y\colon C_AUY\to UC_{FA}Y$ is a weak equivalence for all~$Y$.
\item[{\rm (iv)}]
$C_{FA}$ is a lift of $C_A$ to $\Ho(\M^T)$.
\end{itemize}
\end{theorem}

\begin{proof}
If $(C_A,c)$ lifts to $\Ho(\M^T)$, then there is a colocalization functor
\[
C^T\colon \Ho(\M^T)\longrightarrow\Ho(\M^T)
\]
with a natural transformation $c^T\colon C^T\to {\rm Id}$ and a natural isomorphism $h\colon C_A U\to U C^T$ in $\Ho(\M)$ such that $Uc^T\circ h = cU$. 
By Lemma~\ref{elaprima}, a map $g$ in $\M^T$ is a $C^T$-equivalence if and only if $Ug$ is an $A$-equivalence, and part~(i) of Proposition~\ref{coHoFGf} tells us that $Ug$ is an $A$-equivalence if and only if $g$ is an $FA$-equivalence. Hence $C^T$ and $C_{FA}$ are colocalizations with the same class of equivalences, so there is a natural isomorphism $g\colon C^T\to C_{FA}$ in $\Ho(\M^T)$ with $c\circ g=c^T$.
On the other hand, according to Proposition~\ref{coHoFGf}, for every $T$-algebra $Y$ there is a homotopy unique and homotopy natural $A$-equivalence
\[
\alpha_Y\colon C_AUY\longrightarrow UC_{FA}Y
\] 
such that $Uc_Y\circ\alpha_Y=c_{UY}$, where we need not assume that $Y$ is fibrant since $U$ preserves all weak equivalences.
By its uniqueness up to homotopy, $\alpha_Y$ coincides with $Ug_Y\circ h_Y$ in $\Ho(\M)$ and therefore $\alpha_Y$ is an isomorphism for every $T$-algebra $Y$, thus yielding~(iii).

We next prove that (ii) and (iii) are equivalent. Indeed, $\alpha_Y$ is a weak equivalence if and only if $UC_{FA}Y$ is $A$\nobreakdash-cell\-ular. Hence $\alpha_Y$ is a weak equivalence for all $Y$ if and only if $U$ sends $FA$\nobreakdash-cell\-ular objects to $A$-cellular objects, as claimed.

Finally, (iii) implies (iv) by definition, and (iv) implies~(i).
\end{proof}

As one might have expected, the necessary and sufficient condition (ii) for the existence of a lifting of a cellularization $C_A$ to $\Ho(\M^T)$ given in Theorem~\ref{cofumfumfum} is stronger than the one given in Theorem~\ref{easy} for the existence of a lifting of $C_A$ to $\Ho(\M)^T$. 

\begin{corollary}
\label{cotwolifts}
If $C_A$ lifts to $\Ho(\M^T)$ then $C_A$ also lifts to $\Ho(\M)^{T}$.
\end{corollary}

\begin{proof}
If $C_A$ lifts to $\Ho(\M^T)$ then, by Theorem~\ref{cofumfumfum}, $U$ sends $FA$-cellular objects to $A$\nobreakdash-cell\-ular objects. 
According to part~(ii) of Proposition~\ref{coHoFGf}, $F$ sends cofibrant $A$-cellular objects to $FA$\nobreakdash-cell\-ular objects. Therefore, $T=UF$ sends cofibrant $A$-cellular objects to $A$-cellular objects, and consequently $C_A$ lifts to $\Ho(\M)^{T}$ by part~(b) of Theorem~\ref{easy}.
\end{proof}

\begin{corollary}
\label{CACTA}
If $A$ and $TA$ are cofibrant and both $C_A$ and $C_{TA}$ lift to $\Ho(\M^T)$, then there is a natural isomorphism
\[
C_A UY\cong C_{TA} UY
\]
in $\Ho(\M)$ for all $T$-algebras $Y$.
\end{corollary}

\begin{proof}
According to Theorem~\ref{cofumfumfum}, our assumptions yield natural isomorphisms
\[
C_AUY\cong UC_{FA}Y, \qquad C_{TA}UY\cong UC_{FTA}Y
\]
in $\Ho(\M)$ for all $T$-algebras $Y$. The fact that $C_A$ lifts to $\Ho(\M^T)$ implies, by Theorem~\ref{cofumfumfum}, that $U$ sends $FA$-cellular objects to $A$-cellular objects. Hence $TA=UFA$ is $A$-cellular, and Proposition~\ref{FAFTAFfFTf} implies then that the classes of $FA$-cellular objects and $FTA$-cellular objects coincide. This means that $C_{FA}\cong C_{FTA}$, hence completing the argument.
\end{proof}

A favorable situation for the use of Corollary~\ref{CACTA} is when the forgetful functor $U$ has a right adjoint. An important case when this happens will be discussed in the next section.

\begin{proposition}
\label{uf}
Suppose that the forgetful functor $U\colon\M^T\to\M$ has a Quillen right adjoint.
If $TA$ is $A$-cellular, then $C_A$ and $C_{TA}$ lift to $\Ho(\M^T)$ and there is a natural iso\-morphism
\[
C_A UY\cong C_{TA} UY
\]
in $\Ho(\M)$ for all $T$-algebras $Y$.
\end{proposition}

\begin{proof}
Since both $F$ and $U$ are Quillen left adjoints and $A$ is assumed to be cofibrant, $FA$ and $TA=UFA$ are cofibrant. Moreover, the fact that $U$ is a Quillen left adjoint implies, by part~(ii) of Proposition~\ref{coHoFGf}, that $U$ sends cofibrant $FA$-cellular objects to $TA$\nobreakdash-cell\-ular objects, and, since $U$ preserves weak equivalences, $U$ sends all $FA$-cellular objects to $TA$\nobreakdash-cellular ones. Now the assumption that $TA$ is $A$-cellular implies, by Remark~\ref{cellular}, that $TA$-cellular objects are $A$-cellular and therefore $U$ sends $FA$-cellular objects to $A$-cellular objects.
Therefore $C_A$ lifts to $\Ho(\M^T)$ by Theorem~\ref{cofumfumfum}. Moreover, according to part~(a) of Proposition~\ref{FAFTAFfFTf}, the classes of $FA$-cellular objects and $FTA$-cellular objects coincide. Hence $U$ sends $FTA$\nobreakdash-cellular objects to $TA$-cellular objects, and this implies that $C_{TA}$ also lifts to $\Ho(\M^T)$, so Corollary~\ref{CACTA} applies.
\end{proof}

We next discuss the case of homotopical localizations, which is largely similar.

\begin{theorem}
\label{fumfumfum}
Let $T=UF$ be the Eilenberg--Moore factorization of a monad $T$ on a model category $\M$ such that the category $\M^T$ of $T$-algebras has a transferred model structure. Let $f$ be a map in $\M$ between cofibrant objects such that the left Bousfield localizations $\M_f$ and $(\M^T)_{Ff}$ exist. Then the following statements are equivalent:
\begin{itemize}
\item[{\rm (i)}]
$L_f$ lifts to $\Ho(\M^T)$.
\item[{\rm (ii)}]
$U$ sends $Ff$-equivalences to $f$-equivalences.
\item[{\rm (iii)}]
The comparison map
$\beta_Y\colon L_fUY\to UL_{Ff}Y$
is a weak equivalence for all~$Y$.
\item[{\rm (iv)}]
$L_{Ff}$ is a lift of $L_f$ to $\Ho(\M^T)$.
\end{itemize}
\end{theorem}

\begin{proof}
Recall from Proposition~\ref{HoFGf} that for every $T$-algebra $Y$ there is a homotopy unique and homotopy natural map
\[
\beta_Y\colon L_fUY\longrightarrow UL_{Ff}Y
\] 
such that $\beta_Y\circ l_{UY}=Ul_Y$, and $\beta_Y$ is a weak equivalence if and only if  $Ul_Y\colon UY\to UL_{Ff}Y$ is an $f$-equivalence. Hence $\beta_Y$ is a weak equivalence for all $Y$ if and only if $U$ sends all $Ff$\nobreakdash-loc\-alization maps to $f$-equivalences, and this happens if and only if $U$ sends $Ff$\nobreakdash-equiv\-alences to $f$-equivalences. 
This proves that (ii) $\Leftrightarrow$ (iii).

The claim (i) means that there is a localization functor
\[
L^T\colon \Ho(\M^T)\longrightarrow\Ho(\M^T)
\]
with a natural transformation $l^T\colon {\rm Id}\to L^T$ and a natural isomorphism $h\colon L_f U\to U L^T$ in $\Ho(\M)$ such that $h\circ lU = U l^T$.
By Lemma~\ref{elaprima}, a $T$-algebra $Y$ is $L^T$-local if and only if $UY$ is $f$-local, and part~(i) of Proposition~\ref{HoFGf} implies that $UY$ is $f$-local if and only if $Y$ is $Ff$\nobreakdash-local (no fibrancy assumption on $Y$ is necessary, since $U$ preserves all weak equivalences). Hence $L^T$ and $L_{Ff}$ are localizations with the same essential image class, so 
there is a natural isomorphism $g\colon L^T\to L_{Ff}$ in $\Ho(\M^T)$ with $g\circ l^T=l$.
By its uniqueness up to homotopy, $\beta_Y$ coincides with $Ug_Y\circ h_Y$ in $\Ho(\M)$ and hence $\beta_Y$ is an isomorphism for every $T$-algebra $Y$, which proves~(iii).
By definition (iii) implies (iv), and (iv) implies (i).
\end{proof}

\begin{obs}
\label{toostrong}
The assumption that the model structure $(\M^T)_{Ff}$ exists is unnecessarily strong in the statement of Theorem~\ref{fumfumfum}. The arguments given in the proof remain valid to conclude that (i) and (ii) are equivalent if we only assume that $L_{Ff}$ exists on $\Ho(\M^T)$, since the existence of $\beta_Y$ is inferred in Proposition~\ref{HoFGf} by means of the model structure~$\M_f$.
\end{obs}

In the next statements we assume that left Bousfield localizations of $\M$ and $\M^T$ exist (with the weakening mentioned in Remark~\ref{toostrong} when pertinent). Proofs are omitted since they are similar to the ones given in the case of colocalizations.

\begin{corollary}
\label{twolifts}
If $L_f$ lifts to $\Ho(\M^T)$ then $L_f$ also lifts to $\Ho(\M)^{T}$.
\end{corollary}

\begin{corollary}
\label{coCACTA}
Suppose that the domains and codomains of $f$ and $Tf$ are cofibrant. 
If both $L_f$ and $L_{Tf}$ lift to $\Ho(\M^T)$, then there is a natural isomorphism
\[
L_f UY\cong L_{Tf} UY
\]
in $\Ho(\M)$ for all $T$-algebras $Y$.
\end{corollary}

\begin{proposition}
\label{couf}
Suppose that the forgetful functor $U\colon\M^T\to\M$ has a Quillen right adjoint.
If $Tf$ is an $f$-equivalence, then $L_f$ and $L_{Tf}$ lift to $\Ho(\M^T)$ and there is a natural isomorphism
\[
L_f UY\cong L_{Tf} UY
\]
in $\Ho(\M)$ for all $T$-algebras $Y$.
\end{proposition}

\begin{proof}
The domain and codomain of $Tf$ are cofibrant since $U$ preserves cofibrations and hence so does~$T$. By part~(b) of Proposition~\ref{FAFTAFfFTf}, the class of $FTf$-equivalences coincides with the class of $Ff$-equivalences. Hence, in order to infer that $L_{Tf}$ lifts to $\Ho(\M^T)$, we need to prove that $U$ sends $Ff$-equivalences between cofibrant objects to $Tf$\nobreakdash-equiv\-alences. This follows from part~(ii) of Proposition~\ref{HoFGf} since $U$ is, by assumption, a Quillen left adjoint. Hence our claim is implied by Corollary~\ref{coCACTA}.
\end{proof}

There is, however, a relevant distinction between right Bousfield localizations and left Bousfield localizations of categories of $T$-algebras equipped with transferred model structures, due to the fact that in a right-induced model structure one asks the forgetful functor to create fibrations but not cofibrations. Indeed, as shown in Theorem~\ref{MATMTFA} below, the equality of model categories
\begin{equation}
\label{nice}
(\M_A)^T=(\M^T)_{FA}
\end{equation}
always holds, while the analogous equality for $f$-localizations is not necessarily true, and in fact if it holds then $L_f$ lifts to $T$-algebras.

The ideas behind the next theorem were found by Guti\'errez--R\"ondigs--Spitzweck--{\O}stv{\ae}r in \cite{GRSO} for algebras over coloured operads, and also by Batanin--White in~\cite[Theorem~3.4]{BW} and White--Yau in \cite[Corollary~2.8]{WY3} by weakening the assumption that $\M^T$ has a transferred model structure.

\begin{theorem}
\label{MATMTFA}
Let $T=UF$ be the Eilenberg--Moore factorization of a monad $T$ on a model category $\M$ such that $\M^T$ has a transferred model structure. 
\begin{itemize}
\item[{\rm (a)}]
Let $A$ be a cofibrant object in $\M$ such that the right Bousfield localizations $\M_A$ and $(\M^T)_{FA}$ exist. Then $(\M^T)_{FA}$ is a transferred model structure for $(\M_A)^T$.
\item[{\rm (b)}]
Let $f$ be a map in $\M$ between cofibrant objects such that the left Bousfield localizations $\M_f$ and $(\M^T)_{Ff}$ exist. 
If a transferred model structure for $(\M_f)^T$ exists, then this model structure is equal to $(\M^T)_{Ff}$ and in this case $L_f$ lifts to $\Ho(\M^T)$.
\end{itemize}
\end{theorem}

\begin{proof}
For (a), part~(i) of Proposition~\ref{coHoFGf} states that $U\colon (\M^T)_{FA}\to\M_A$ preserves and reflects weak equivalences. Since the fibrations of $(\M^T)_{FA}$ are those of $\M^T$ and the fibrations of $\M_A$ are those of~$\M$, we have that $U\colon (\M^T)_{FA}\to\M_A$ also preserves and reflects fibrations, and this means that $(\M^T)_{FA}$ is a transferred model structure for~$(\M_A)^T$.

As for (b), in the adjunction
\begin{equation}
\label{badone}
F:\M_f\rightleftarrows(\M^T)_{Ff}:U
\end{equation}
the cofibrations and trivial fibrations in $\M_f$ are those of $\M$ and the cofibrations and trivial fibrations in $(\M^T)_{Ff}$ are those of $\M^T$; hence $F$ preserves cofibrations and $U$ preserves and reflects trivial fibrations. Moreover, part~(i) of Proposition~\ref{HoFGf} says that $U$ preserves and reflects fibrant objects, while part~(ii) tells us that $F$ preserves weak equivalences between cofibrant objects. 

If a transferred model structure exists for $(\M_f)^T$, then it has the same trivial fibrations and the same fibrant objects as $(\M^T)_{Ff}$, since these are created by $U$ in both cases. Hence they also have the same cofibrations ---and thus the same cofibrant objects--- and the same weak equivalences between cofibrant objects, since these are determined by the fibrant objects. Therefore they have the same weak equivalences (since cofibrant approximations are trivial fibrations) and this proves that, indeed, $(\M_f)^T=(\M^T)_{Ff}$ as model categories.
From this fact it follows that $U\colon\M^T\to\M$ sends $Ff$-equivalences to $f$-equivalences, which, according to Theorem~\ref{fumfumfum}, implies that $L_f$ lifts to $\Ho(\M^T)$. 
\end{proof}

Consequently, in order to display a case where a transferred model structure for $(\M_f)^T$ does not exist, it is enough to exhibit an $f$-localization that does not lift to $T$\nobreakdash-algebras. 
This is done at the end of the next section.

\section{Module spectra}
\label{modules}

The category $\Spectra$ of symmetric spectra over simplicial sets has a closed symmetric monoidal structure with internal function spectrum $\Hom_S(-,-)$ defined as in \cite[Definition~2.2.9]{HSS}. We consider the stable model structure \cite[Theorem~3.4.4]{HSS} on~$\Spectra$, which, according to \cite[Theorem~4.2.5]{HSS}, is Quillen equivalent to the Bousfield--Friedlander model structure on the category of ordinary spectra. If $Q$ is a cofibrant replacement functor and $R$ is a fibrant replacement functor on $\Spectra$, we call $\DF(X,Y)=\Hom_S(QX,RY)$ a \emph{derived function spectrum}.

At the same time, $\Spectra$ is a simplicial category with enrichment defined as 
\[
\Map(X,Y)_n=\Spectra(X\wedge\Delta[n]_+,Y)
\] 
for all~$n$, where $\Delta[n]_+$ denotes the standard $n$-simplex with a disjoint basepoint. Thus, $\Map(QX,RY)$ is a convenient choice of a homotopy function complex from $X$ to $Y$ in~$\Spectra$. Moreover, it follows from \cite[Lemma~6.1.2]{Hovey} or \cite[Corollary~2.2.11]{HSS} that the homotopy groups of the simplicial set $\Map(QX,RY)$ are isomorphic to those of the spectrum $\DF(X,Y)$ in non-negative dimensions, hence to those of the connective cover $\DF^c(X,Y)$ in all dimensions. As a consequence of this fact, homotopy orthogonality \eqref{enrichedorthogonal} in $\Spectra$ can be formulated in terms of $\DF^c(-,-)$, as done in \cite{Bo96,Bo99,CG,Javier3} and in many other articles.

For any ring spectrum~$E$, the category $\EMod$ of left $E$\nobreakdash-module spectra admits, by \cite[Corollary~5.4.2]{HSS}, a transferred model structure for the Eilenberg--Moore factorization
\[
F : \Spectra \rightleftarrows \EMod : U
\]
of the monad $TX=E\wedge X$ on~$\Spectra$. Here $T$ preserves cofibrant objects if we impose that $E$ be cofibrant as a spectrum, and $T$ preserves colimits since it is left adjoint to $\Hom_S(E,-)$ as $\Spectra$ is closed symmetric monoidal.

Furthermore, if $E$ is cofibrant, then $\Hom_S(E,-)$ with values in $\EMod$ is Quillen right adjoint to the forgetful functor~$U$, since for every left $E$-module $M$ and every spectrum~$Y$ there are natural bijections
\[
\Spectra(UM,Y)\cong\Spectra(E\wedge_EM,Y)\cong\EMod(M,\Hom_S(E,Y)).
\]
This appears in \cite[Lemma~III.6.5(ii)]{EKMM} and it also follows from \cite[Theorem~1.5(iii)]{Schwede}, since $E\wedge -$ preserves level cofibrations. Hence $U$ also preserves colimits.

The model category $\EMod$ of left $E$-modules is cofibrantly generated by \cite[Theorem~4.1]{SS}, locally presentable by \cite[\S\,2.78]{AR}, and proper since $U$ is both a Quillen left adjoint and a Quillen right adjoint. Hence, left and right Bousfield localizations exist on $\EMod$ as well as on~$\Spectra$.

We say that a homotopy functor on spectra is \emph{triangulated} if it preserves cofibre sequences. For a localization or a cellularization, this condition is equivalent to commuting with suspension, as shown in \cite[Theorem~2.7]{CG} and \cite[Theorem~2.9]{Javier4}.
Localizations and cellularizations need not be triangulated. However, $C_A$ is triangulated if $A=\bigvee_{i=0}^{\infty} \Sigma^{-i} W$ for some spectrum~$W$ ---this condition was also considered in~\cite[Lemma~5.7]{BR}---, and $L_f$ is triangulated if $f=\bigvee_{i=0}^{\infty} \Sigma^{i}g$ for some map~$g$. Localizations with respect to homology theories \cite{BoLoc} are triangulated.

\begin{theorem}
\label{mainone}
Let $A$ and $E$ be cofibrant symmetric spectra and suppose that $E$ is a ring spectrum. Let $U$ be the forgetful functor from left $E$-modules to spectra. If $E$ is connective or $C_A$ is triangulated, then $C_A$ lifts to left $E$-modules and
\[
C_AUY\simeq C_{U(E\wedge A)}UY\simeq UC_{E\wedge A}Y
\]
for every left $E$-module $Y$.
\end{theorem}

\begin{proof}
Since $U$ is a Quillen left adjoint, we are ready to apply Proposition~\ref{uf}, and it will suffice to show that $U(E\wedge A)$ is $A$-cellular.
For this we need to assume, as in \cite[\S\,4]{Javier2} and \cite[Theorem~4.1]{Javier4}, that $E$ is connective or $C_A$ is triangulated. If $E$ is connective, then we may use the connective cover $\DF^c(-,-)$ of the derived function spectrum instead of a homotopy function complex to infer that every $A$-equivalence $g\colon X\to Y$ induces
\begin{align*}
\DF^c(E\wedge A,X) & \simeq \DF^c(E,\DF(A,X)) \simeq \DF^c(E,\DF^c(A,X)) \\ & \simeq \DF^c(E,\DF^c(A,Y))\simeq \DF^c(E,\DF(A,Y))\simeq \DF^c(E\wedge A,Y),
\end{align*}
since $g$ induces $\DF^c(A,X)\simeq \DF^c(A,Y)$. 
If $E$ is not necessarily connective but $C_A$ is triangulated, then each $A$-equivalence $g\colon X\to Y$ induces a weak equivalence $\DF(A,X)\simeq \DF(A,Y)$ by \cite[Theorem~2.9]{Javier4}, and therefore
\[
\DF(E\wedge A,X)\simeq \DF(E,\DF(A,X))\simeq \DF(E,\DF(A,Y))\simeq \DF(E\wedge A,Y),
\]
as needed.
In conclusion, if $E$ is connective or $C_A$ is triangulated, then $C_A$ lifts to $\Ho(\EMod)$ and Proposition~\ref{uf} tells us that
$C_AUY\cong C_{U(E\wedge A)}UY$
for every cofibrant spectrum $A$ and every left $E$-module $Y$. Moreover Theorem~\ref{cofumfumfum} implies that $C_{E\wedge A}$ is a lift of $C_A$ and hence
$C_AUY\simeq UC_{E\wedge A}Y$ 
for every left $E$-module $Y$, as claimed.
\end{proof}

\begin{ex}
As an easy example, pick $A=S$ (the sphere spectrum), so that the $S$-cellular spectra are precisely the connective ones. It then follows from Theorem~\ref{mainone} that, if $E$ is a connective ring spectrum, then $C_EX\simeq C_SX\simeq X^c$ for every spectrum $X$ underlying a left $E$-module. Hence, connective $E$-modules are $E$-cellular.
\end{ex}

\begin{theorem}
\label{maintwo}
Let $E$ be a cofibrant symmetric spectrum and let $f$ be a map between cofibrant spectra. Suppose that $E$ is a ring spectrum, and let $U$ be the forgetful functor from left $E$\nobreakdash-mod\-ules to spectra. If $E$ is connective or $L_f$ is triangulated, then $L_f$ lifts to left $E$-modules and
\[
L_fUY\simeq L_{U(E\wedge f)}UY\simeq UL_{E\wedge f}Y
\]
for every left $E$-module $Y$.
\end{theorem}

\begin{proof}
As in Theorem~\ref{mainone}, we use the fact that the forgetful functor $U$ has a Quillen right adjoint, namely $\Hom_S(E,-)$.
In view of Proposition~\ref{couf}, it is enough to prove that $U(E\wedge f)$ is an $f$-equivalence.
This happens provided that $E$ is connective or $L_f$ is triangulated, by a similar argument as in the proof of Theorem~\ref{mainone}; see also \cite[Theorem~2.7]{CG}. Hence, under these assumptions, $L_f$ lifts to $\Ho(\EMod)$ and Proposition~\ref{couf} together with Theorem~\ref{fumfumfum} imply that
\[
L_fUY\simeq L_{U(E\wedge f)}UY\simeq UL_{E\wedge f}Y
\]
for every left $E$-module $Y$, which is the corresponding homotopical version of~\eqref{Sotimesf}.
\end{proof}

The next counterexample is based on \cite[Example~4.4]{CG} and shows that the hypotheses made in Theorem~\ref{mainone} and Theorem~\ref{maintwo} are necessary.

\begin{ex}
\label{shocking}
Let $K(n)$ be the $n$th Morava $K$-theory spectrum at a prime $p$ for any $n\ge 1$, and let $k(n)$ be its connective cover. Then $H\ZZ/p\wedge K(n)=0$ while $H\ZZ/p\wedge k(n)\ne 0$, as proved in \cite[Theorem~2.1]{Ravenel}. This implies that $k(n)$ is not a homotopy retract of $K(n)\wedge k(n)$ and therefore $k(n)$ cannot be a left $K(n)$-module. Here $k(n)\simeq C_SK(n)$ where $S$ denotes the sphere spectrum; consequently, $C_S$ does not lift to $K(n)$-modules.
This counterexample shows that a cellularization $C_A$ of spectra need not lift to $E$-modules if the assumptions that $E$ is connective or $C_A$ is triangulated in Theorem~\ref{mainone} both fail to hold. Note that $C_S$ is not triangulated since it converts the cofibre sequence $\Sigma^{-1}S\to 0\to S$ into $0\to 0\to S$.

Similarly, if one removes the assumption that $E$ is connective or $L_f$ is triangulated from Theorem~\ref{maintwo}, then it need no longer be true that $L_f$ lifts to $E$-module spectra. Indeed, the cofibre of the canonical map $k(n)\to K(n)$ is the Postnikov section $P_{-1}K(n)$;
since $ H\ZZ/p\wedge P_{-1}K(n)\ne 0$, we find that $P_{-1}K(n)$ is not a homotopy retract of $K(n)\wedge P_{-1}K(n)$ and this implies that $P_{-1}K(n)$ cannot be a left $K(n)$-module. But $P_{-1}K(n)=L_fK(n)$ with $f\colon S\to 0$.
Thus, if $f\colon S\to 0$ and $E=K(n)$, then $L_f$ does not lift to $E$-modules.
\end{ex}

Example~\ref{shocking} also shows that a transferred model structure for $(\M_f)^T$ need not exist for a monad $T$ acting on a model category~$\M$. This is a consequence of part~(b) of Theorem~\ref{MATMTFA}, since in Example~\ref{shocking} we have exhibited a case where $L_f$ does not lift to~$\Ho(\M^T)$. 

\section{Loop spaces and infinite loop spaces}
\label{operads}

In this section we consider simplicial operads (that is, operads taking values in simplicial sets) acting on pointed simplicial sets endowed with the Cartesian product. 
This choice is due to the fact that important monads such as the James construction, the $Q$-construction, or the infinite symmetric product, which will be used in this section, involve products and basepoints in their definition. 

Another reason is that in the unpointed category cellularizations are trivial due to the fact that every space $X$ is a retract of $\map(A,X)$ if $A$ is nonempty, as pointed out in \cite[Remark~2.A.1.1]{DF} and \cite[Remark~3.1.10]{Hirschhorn}. In the case of localizations, there are no essential differences between working with basepoints or without them. In other words, $\map_*(-,-)$ or $\map(-,-)$ can indistinctly be used to test homotopy orthogonality for localizations of pointed connected spaces \cite[Lemma~2.1]{Bo97}, and hence if a map $g$ of pointed spaces is an $f$-equivalence then both $g\times K$ and $g\wedge K$ are $f$-equivalences for every pointed simplicial set~$K$. However, if a pointed space $X$ is $A$-cellular then $X\wedge K$ is $A$-cellular for all $K$ but $X\times K$ need not be (since it has $K$ as a retract).

For a simplicial operad $P$ and every~$n$, the (unpointed) simplicial set $P(n)$ is equipped with an action of the symmetric group~$\Sigma_n$. We assume, as usual, the existence of a unit element $u\in P(1)$, and we will also assume that operads are \emph{reduced}, meaning that $P(0)$ is a single point, denoted by~$*$. 

If $P$ is such an operad and $X$ is a simplicial set with a basepoint $x_0$ then a \emph{$P$-algebra structure} on $X$ is a morphism of operads $P\to\End(X)$ where $\End(X)(n)$ is the based function complex $\map_*(X^n,X)$, as in \cite[\S\,1]{May2009}. 
Here we denote by $X^n$ the $n$\nobreakdash-fold Cartesian product of $X$ with itself, which is meant to be a point if $n=0$, and pick $(x_0,\dots,x_0)$ as basepoint of $X^n$ if $n\ne 0$. The operad $\End(X)$ is reduced and the image of the unit element $u\in P(1)$ 
under the structure map $P(1)\to\map_*(X,X)$ is assumed to be the identity map.

For a simplicial operad $P$ acting on pointed simplicial sets through basepoint-preserving maps, as explained in \cite{May2009}, the $P$-algebras coincide with the algebras over the \emph{reduced} monad $T$ where $TX$ is defined, for a pointed simplicial set $X$, as the quotient of
\begin{equation}
\label{operad2}
\coprod_{n\ge 0} P(n)\times_{\Sigma_n}X^{n}
\end{equation}
by identifiying 
$(w,s_iy)\in P(n)\times X^{n}$
with
$(\sigma_i w,y)\in P(n-1)\times X^{n-1}$
for $1\le i\le n$ and all $w\in P(n)$ and $y\in X^{n-1}$. Here the map $s_i\colon X^{n-1}\to X^{n}$ inserts the basepoint into the $i$-th place and $\sigma_i\colon P(n)\to P(n-1)$ is the $i$-th degeneracy, defined as 
\[
\sigma_iw=\gamma_i(w,u,\dots,u,*,u,\dots,u), 
\]
where
\[
\gamma_i\colon
P(n)\times P(1)^{i-1}\times P(0)\times P(1)^{n-i}\longrightarrow P(n-1)
\] 
is among the multiplication maps of the operad $P$; see \cite[\S\,4]{May2009}. Thus,
in particular, $P(0)$ is identified with $(w,x_0,\dots,x_0)\in P(n)\times X^n$ for every $n$ and all $w\in P(n)$. The unit map $\eta_X\colon X\to TX$ of the monad sends each element $x$ to $(u,x)\in P(1)\times X$. 

The same happens for non-symmetric operads by discarding the $\Sigma_n$-action from~\eqref{operad2}. For example, if $\Ass$ is the (non-symmetric) unital associative operad, for which $A(n)$ is a single point for all~$n$, then the associated reduced monad on pointed simplicial sets is the James construction \cite{James}.

\begin{lemma}
\label{great}
If $T$ is the reduced monad associated with a simplicial operad~$P$ acting on pointed simplicial sets, then $T$ preserves $f$\nobreakdash-equiv\-alences for every map~$f$.
\end{lemma}

\begin{proof}
For a pointed simplicial set $X$, we write $TX$ as a homotopy colimit of the partial sums 
\[
T_kX=\coprod_{0\le n\le k}P(n)\times_{\Sigma_n}X^n
\]
subject to the same identifications as in \eqref{operad2}. Thus, $T_0X=P(0)$, $T_1X\cong P(1)_+\wedge X$, and $T_kX$ is a pointed simplicial set such that
\[
T_kX/T_{k-1}X\cong P(k)_+\wedge_{\Sigma_k}X^{\wedge k},
\]
where $X^{\wedge k}$ denotes $X\wedge\cdots\wedge X$ with $k$ factors.
Hence for each $k$ and every pointed simplicial set $Z$ there is a Kan fibre sequence
\begin{equation}
\label{fibration}
\map_*(P(k)_+\wedge_{\Sigma_k}X^{\wedge k},Z)\longrightarrow
\map_*(T_{k}X,Z)\longrightarrow
\map_*(T_{k-1}X,Z).
\end{equation}

Since the smash product of two $f$-equivalences is an $f$-equivalence, if $g\colon X\to Y$ is an $f$-equivalence then so is the induced map
\begin{equation}
\label{inducedmap}
W\wedge X\wedge \cdots \wedge X\longrightarrow
W\wedge Y\wedge \cdots  \wedge Y
\end{equation}
for every finite number of factors and every pointed simplicial set~$W$, where \eqref{inducedmap} is the identity on the first factor and $g$ on the other factors.

If the operad $P$ is non-symmetric then we can omit the $\Sigma_n$-action and \eqref{fibration} 
proves inductively that if $g\colon X\to Y$ is an $f$-equivalence then $T_kg$ is an $f$-equivalence for every $k$ and consequently $Tg$ is also an $f$-equivalence.

If $P$ is symmetric, we need to use the fact that the quotient $P(k)_+\wedge_{\Sigma_k}X^{\wedge k}$ is a colimit over $\Sigma_k$ (viewed as a small category), and it is also a homotopy colimit if $\Sigma_k$ acts freely on $P(k)$, but not otherwise. However, $P(k)_+\wedge_{\Sigma_k}X^{\wedge k}$ is a homotopy colimit of a (free) diagram indexed by the opposite of the orbit category of $\Sigma_k$, where the value of the diagram at $\Sigma_k/G$ is the fixed-point space 
\[
P(k)^G_+\wedge (X^{\wedge k})^G;
\] 
cf.\,\cite[\S\,4.A.4]{DF}. Since each space $(X^{\wedge k})^G$ is homeomorphic to a product $X^{\wedge m}$ with $m\le k$, we obtain that the map $Tg\colon TX\to TY$ is a homotopy colimit of $f$-equivalences, and therefore it is itself an $f$-equivalence.
\end{proof}

Essentially the same argument works for cellularizations.
We need the assumption that $A$ be connected in order to avoid triviality, since $S^0$ is a retract of every non-connected simplicial set and all simplicial sets are $S^0$-cellular.

\begin{lemma}
\label{cogreat}
If $T$ is the reduced monad associated with a simplicial operad~$P$ acting on pointed simplicial sets and $A$ is connected, then $T$ preserves $A$\nobreakdash-cell\-ular simplicial sets.
\end{lemma}

\begin{proof}
If $X$ is $A$-cellular, then every finite smash product
$
P(k)_+\wedge X^{\wedge k}
$
is $A$-cellular for $k\ge 1$ by \cite[Theorem~2.D.8]{DF}, while $P(0)=*$ by assumption, which is also $A$-cellular. The proof continues similarly as in Lemma~\ref{great}.
\end{proof}

\begin{ex}
The infinite symmetric product \cite{DT}, denoted by $SP^{\infty}$, is the reduced monad associated with the commutative operad.
Its algebras are the commutative monoids. The argument given in the proof of Lemma~\ref{great} was used in \cite[Theorem~1.3]{CRT} to prove that $SP^{\infty}$ preserves $f$\nobreakdash-equiv\-alences for every map~$f$. Moreover, it was shown in \cite[Proposition~1.1]{CRT} that a space $X$ underlies an $SP^{\infty}$-algebra in the pointed homotopy category if and only if $X$ is a GEM, i.e., a weak product of abelian Eilenberg--Mac\,Lane spaces; moreover, in this case the $SP^{\infty}$-algebra structure on $X$ is unique up to isomorphism. Therefore, by Lemma~\ref{great} and part~(a) of Theorem~\ref{easy}, every $f$-localization preserves GEMs (as first shown in \cite[Chapter~4]{DF}) and defines in fact a localization on the homotopy category of GEMs. Similarly, every $A$-cellularization preserves GEMs and defines a colocalization on them by Lemma~\ref{cogreat} and part~(b) of Theorem~\ref{easy}.
\end{ex}

\begin{obs}
Lemma~\ref{great} also holds for the \emph{unreduced} monad $\tilde T$ associated with a simplicial operad~$P$, that is,
\[
\tilde TX=\coprod_{n\ge 0} P(n)\times_{\Sigma_n}X^{n}
\]
without any basepoint identifications, whose algebras are also the $P$-algebras \cite[\S\,4]{May2009}. The proof proceeds with the same argument as in the proof of Lemma~\ref{great}, using the fact that if $X\to Y$ is an $f$-equivalence then $W\times X^n\to W\times Y^n$ is also an $f$-equivalence for every simplicial set $W$ and all~$n$.
However, Lemma~\ref{cogreat} is not true for unreduced monads, since, for an $A$-cellular space $X$, a product $W\times X$ is not $A$-cellular unless $W$ is itself $A$-cellular.
Still, Lemma~\ref{cogreat} holds for unreduced monads if we assume that each $P(n)$ and all fixed-point spaces $P(n)^G$ are contractible or empty for every subgroup $G\subseteq\Sigma_n$. This condition is trivially satisfied if $P(n)$ is a single point for all~$n$ and also if each $P(n)$ is contractible and $\Sigma_n$ acts freely on it. This is the case, for instance, for the commutative operad and its cofibrant approximations.
\end{obs}

We next address liftings of localizations $L_f$ and cellularizations $C_A$ to categories of algebras over simplicial operads. The existence of a transferred model structure for such categories of algebras is guaranteed by results in~\cite{BM2}. 
In order to apply Theorem~\ref{cofumfumfum} and Theorem~\ref{fumfumfum} to the associated monad $T=UF$, we need to prove that the forgetful functor $U$ sends $Ff$-equivalences of $T$\nobreakdash-alg\-ebras to $f$-equivalences of spaces for every basepoint-preserving map~$f$, and that $U$ sends $FA$-cellular $T$-algebras to $A$-cellular spaces for every pointed connected space~$A$. Both statements would be straightforward if $U$ preserved homotopy colimits. However, in general, $U$ only preserves \emph{sifted} colimits (including filtered ones and reflexive coequalizers), since $T$ commutes with these due to the fact that sifted colimits commute with finite products of simplicial sets; cf.\,\cite[Proposition~2.5]{JN}.

To surmount this difficulty, we rely on a method used by Guti\'errez--R\"ondigs--Spitzweck--{\O}stv{\ae}r in \cite[\S\,3]{GRSO} with a similar purpose.
For a $T$-algebra $X$, we consider the \emph{standard simplicial resolution} of $X$ by free $T$-algebras (compare with May's two-sided bar construction \cite[\S\,9]{May}), defined as 
\[
B_nX=(FU)^{n+1}X
\] 
for $n\ge 0$. Thus $B_*X$ is a simplicial $T$-algebra with face and degeneracy maps coming from the unit and the counit of the adjunction between $F$ and~$U$.
Moreover, the counit $FUX\to X$ yields a map $B_*X\to X$ of simplicial $T$-algebras, where $X$ is treated as a constant simplicial $T$-algebra.
The monad $T$ commutes with geometric~realization since $T$ is defined by means of a coend formula, namely \eqref{operad2}, and~geom\-etric realization is itself a coend. It then follows that the forgetful functor $U$ also commutes with geometric realization by the argument given in \cite[Proposition~3.12]{JN}. 
This implies that the map $B_*X\to X$ induces a weak equivalence $|B_*X|\simeq X$ of $T$-algebras, since $U$ reflects weak equivalences and the map $U|B_*X|\to UX$ has a left homotopy inverse induced by the unit $UX\to UFUX$; see \cite[Proposition~3.13]{JN}.

\begin{lemma}
\label{summer}
Let $T$ be the reduced monad associated with a simplicial operad acting on pointed simplicial sets, and let $T=UF$ be its Eilenberg--Moore factorization.
\begin{itemize}
\item[{\rm (a)}]
Let $f$ be a basepoint-preserving map and let $\mathcal{G}$ be a collection of maps of $T$-algebras such that $Ug$ is an $f$-equivalence for every $g\in\mathcal{G}$. Then $Uh$ is an $f$-equivalence for every $h$ in the closure of $\mathcal{G}$ under pointed homotopy colimits.
\item[{\rm (b)}]
Let $A$ be a pointed connected simplicial set, and let $\mathcal{D}$ be a collection of $T$-algebras such that $UD$ is $A$-cellular for every $D\in\mathcal{D}$. Then $UX$ is $A$-cellular for every $X$ in the closure of $\mathcal{D}$ under pointed homotopy colimits.
\end{itemize}
\end{lemma}

\begin{proof}
We first prove~(b).
Suppose inductively that $X\simeq\hocolim_{i\in I} D_i$ where $I$ is a small category and $UD_i$ is $A$\nobreakdash-cell\-ular for every~$i$. If we denote $J_n=\hocolim_{i\in I} B_n(D_i)$, then, since geometric realization commutes with colimits,
\[
|J_*|=|\hocolim_{i\in I} B_*(D_i)|\simeq\hocolim_{i\in I} |B_*(D_i)|\simeq \hocolim_{i\in I} D_i\simeq X.
\] 
Moreover, since $U$ preserves weak equivalences and commutes with geometric realization, we find that $UX\simeq U|J_*|\simeq |UJ_*|$.
Hence, in order to infer that $UX$ is $A$-cellular, it is enough to prove that $UJ_n$ is $A$-cellular for all~$n$. Indeed, as $F$ commutes with colimits,
\[
UJ_n= U\hocolim_{i\in I}\, (FU)^{n+1}D_i\simeq UF\hocolim_{i\in I}\, (UF)^{n}UD_i=T\hocolim_{i\in I}\, T^{n}UD_i.
\] 
Since $UD_i$ is $A$-cellular for every $i$ and $T$ preserves $A$-cellular simplicial sets by Lemma~\ref{cogreat}, we conclude that $UJ_n$ is $A$-cellular, as needed. This proves part~(b). 

The proof of part (a) follows the same argument in the category of maps of $T$-algebras, using that $T$ preserves $f$-equivalences by Lemma~\ref{great}. 
\end{proof}

To avoid ambiguity with ``pushouts of maps'', note that the left-hand square
\[
\xymatrix{
V\ar[r]^f
\ar[d]_a & W \ar[d]^b \\
X\ar[r]^g
& Y}
\hspace{2cm}
\xymatrix@C=4pc{
{\rm id}_V\ar[r]^{({\rm id}_V,\,f)}
\ar[d]_{(a,a)} & f \ar[d]^{(a,b)} \\
{\rm id}_X\ar[r]^{({\rm id}_X,\,g)}
& g}
\]
is a pushout square in a category $\mathcal{C}$ if and only if
the right-hand square is a pushout square in the category of arrows of~$\mathcal{C}$.

\begin{theorem}
\label{maybe}
Let $T$ be the reduced monad associated with a simplicial operad $P$ acting on pointed simplicial sets.
\begin{itemize}
\item[{\rm (a)}]
For every basepoint-preserving map $f\colon V\to W$, the functor $L_f$ lifts to $\Ho(\Ssets_*^T)$.
\item[{\rm (b)}]
For every pointed connected simplicial set $A$ the functor $C_A$ lifts to $\Ho(\Ssets_*^T)$.
\end{itemize}
\end{theorem}

\begin{proof}
Let $T=UF$ be the Eilenberg--Moore factorization of~$T$. 
We aim to apply Theorem~\ref{cofumfumfum} and Theorem~\ref{fumfumfum}. The category $\Ssets_*^T$ of $T$-algebras has a transferred model structure by \cite[Proposition~4.1(c)]{BM2}, which is cofibrantly generated by \cite{Crans} and locally presentable by \cite[\S\,2.78]{AR}, since $T$ preserves filtered colimits. It is right proper as $U$ preserves limits, fibrations, and weak equivalences, and consequently right Bousfield localizations $(\Ssets_*^T)_{FA}$ exist for all~$A$. 

It is not known to the authors if left properness of $\Ssets_*^T$ holds in general, although it does if the given operad $P$ is cofibrant, according to \cite[Theorem~4.3]{Spitzweck}, and also under the assumptions made in \cite[Definition~3.1]{Muro3}, which hold for the unital associative operad among others.
However, for our purposes it will be enough that the model category $\Ssets_*^T$ be \emph{Quillen equivalent} to a left proper combinatorial model category, and this is guaranteed by \cite[Corollary~1.2]{Dugger}. This fact ensures that, for every map~$f$, a localization exists on $\Ho(\Ssets_*^T)$ whose equivalences are the $Ff$-equivalences, and this suffices to provide a lifting of $L_f$ to $\Ho(\Ssets_*^T)$ if the remaining assumptions in Theorem~\ref{fumfumfum} are fulfilled (see Remark~\ref{toostrong}).

Thus, for part (a) we need to prove that the forgetful functor $U$ sends $Ff$-equivalences of $T$\nobreakdash-alg\-ebras to $f$-equivalences of spaces, and for part (b) we want to see that $U$ sends $FA$-cellular $T$-algebras to $A$-cellular spaces. 
We treat the latter first. 
The class of $FA$-cell\-ular $T$-algebras is the smallest class of $T$-algebras containing $FA$ and closed under pointed homotopy colimits \cite[\S\,5.5]{Hirschhorn}. Since $UFA=TA$ is $A$-cellular by Lemma~\ref{cogreat}, the result follows from part (b) of Lemma~\ref{summer} with $\mathcal{D}=\{FA\}$.

As for (a), we note that, since $T$ preserves $f$-equivalences by Lemma~\ref{great}, the $Ff$-equiv\-alences coincide with the $FTf$-equivalences by Proposition~\ref{FAFTAFfFTf}.
What we will next prove, for technical convenience, is that $U$ sends $FTf$-equivalences to $f$-equivalences.
In fact it is enough to prove that $U$ sends all localization maps $l_X\colon X\to L_{FTf}X$ to $f$-equivalences, since if $g\colon X\to Y$ is any $FTf$-equivalence then $l_Y\circ g=L_{FTf}(g)\circ l_X$ and $L_{FTf}(g)$ is a weak equivalence.

Each localization map $l_X$ is constructed as a possibly transfinite composite of homotopy pushouts of generating trivial cofibrations in $\Ssets_*^T$ and \emph{horns} on $FTf$ with $n\ge 0$:
\begin{equation}
\label{horn}
\lambda_{n,f}\colon (FTV\otimes \Delta[n])\;\textstyle\coprod_{FTV\otimes\partial\Delta[n]}\;(FTW\otimes\partial\Delta[n])
\longrightarrow FTW\otimes\Delta[n],
\end{equation}
where we are using the simplicial structure of $\Ssets_*^T$ described in \cite[Proposition~2.14]{JN}. Specifically, for a $T$-algebra $X$ with structure map $\alpha\colon TUX\to UX$ and a simplicial set~$K$, the $T$-algebra $X\otimes K$ is defined by means of a reflexive coequalizer:
\[
\xymatrix@1{
F(TUX\times K)\hspace{1.2cm} \ar @/^1pc/ [r]^{F(\alpha\times K)} \ar @/_1pc/ [r]^{\beta} & \hspace{1.2cm} F(UX\times K) \;\; \ar[r] & \;\; X\otimes K
}
\]
where $\beta$ is adjunct to the map $TUX\times K\to T(UX\times K)$ given by the fact that $T$ is a simplicial functor, and the common section is $F(\eta\times K)$ where $\eta\colon UX\to TUX$ is the unit map of~$T$. In particular, $FUY\otimes K\cong F(UY\times K)$ for all $Y$ and~$K$; cf.\,\cite[Lemma~3.8]{GKR}.

Therefore $U(FTf\otimes K)\cong UF(Tf\times K)=T(Tf\times K)$ is an $f$-equivalence for every $K$, since $T$ preserves $f$-equivalences by Lemma~\ref{great}.

Then Lemma~\ref{summer} tells us that $Uh$ is an $f$-equivalence for every map $h$ that can be constructed from maps of the form $FTf\otimes K$ by means of homotopy pushouts. This implies that $U(\lambda_{n,f})$ is an $f$-equivalence for every horn $\lambda_{n,f}$, since $\lambda_{n,f}\circ g=FTf\otimes\Delta[n]$ with $g$ a pushout of $FTf\otimes\partial\Delta[n]$ along the cofibration $FTV\otimes\partial\Delta[n]\to FTV\otimes\Delta[n]$. Finally, as $l_X$ can be constructed from horns and weak equivalences by means of homotopy pushouts, we conclude that $Ul_X$ is an $f$-equivalence using Lemma~\ref{summer} again.
\end{proof}

In what follows, $\Omega$ denotes the derived loop functor on pointed simplicial sets, that is, $\Omega X=\map_*(\mathbb{S}^1,RX)$, where $\mathbb{S}^1=\Delta[1]/\partial\Delta[1]$ and $R$ is a fibrant replacement functor.
Let $\Ass$ be the (non-symmetric) unital associative operad, for which $\Ass(n)$ is a single point for all~$n$, and let $\varphi\colon A_{\infty}\to\Ass$ be a cofibrant resolution \cite{Muro1,Muro2}.
If we let $\Ass$ act on pointed simplicial sets then the corresponding algebras are the monoids and the associated reduced monad is the James  construction \cite{James}.

The morphism of operads $\varphi\colon A_{\infty}\to\Ass$ yields a Quillen equivalence 
\[
\varphi_{!} : A_{\infty}\text{-alg}\rightleftarrows A\text{-alg}:\varphi^*,
\]
so the homotopy category of monoids is equivalent to the homotopy category of $A_{\infty}$-algebras in pointed simplicial sets.
Moreover, the classifying space functor $B$ is part of an adjunction
\begin{equation}
\label{BOmega}
B:\Ho(A_{\infty}\text{-alg})\rightleftarrows\Ho(\Ssets_*):\Omega,
\end{equation}
which restricts, as a special case of \eqref{Leinster}, to an equivalence of categories between the full subcategory of $\Ho(A_{\infty}\text{-alg})$ whose objects are those $M$ such that the unit $\eta_M\colon M\to \Omega BM$ is an isomorphism (that is, grouplike spaces) and the full subcategory of connected simplicial sets, which are precisely those $X$ for which the counit $\varepsilon_X\colon B\Omega X\to X$ is an isomorphism. 

\begin{corollary}
\label{loops}
If $f$ is any basepoint-preserving map between connected simplicial sets, then
\begin{itemize}
\item[{\rm (i)}]
$L_f\Omega X\simeq \Omega L_{\Sigma f}X$, and
\item[{\rm (ii)}]
$L_f\Omega X\simeq L_{\Omega\Sigma f}\Omega X$
\end{itemize}
for all pointed simplicial sets $X$.
\end{corollary}

\begin{proof}
Let
$
F:\Ssets_*\rightleftarrows \Ass\text{-alg}: U
$
be the Eilenberg--Moore factorization of the James cons\-truc\-tion $J$ as a monad on pointed simplicial sets. 
Theorem~\ref{maybe} tells us that $L_f$ lifts to $\Ho(\Ass\text{-alg})$, and we then infer from the equivalence of (i) and (iii) in Theorem~\ref{fumfumfum} that $L_fU M\simeq U L_{F f}M$ for every monoid~$M$.

Let $X$ be any pointed simplicial set.
Since $\varphi\colon A_{\infty}\to \Ass$ allows rectification of algebras, 
there is a monoid $M_X$ whose underlying space is weakly equivalent to $\Omega X$.
Hence,
\[
L_f \Omega X \simeq L_f UM_X\simeq U L_{Ff} M_X.
\]

Next we observe that, since homotopical localizations preserve $\pi_0$ by Example~\ref{connective}, the functor $L_{Ff}$ restricts to the full subcategory of $\Ho(A_{\infty}\text{-alg})$ whose objects are grouplike and $L_{BFf}$ restricts to the full subcategory of connected simplicial sets. Since \eqref{BOmega} sets up an equivalence between these two categories, $B$ preserves both local objects and equivalences and, by Theorem~\ref{alfabeta}, there is a comparison map of $\beta$ type
\[
L_{Ff} M_X\longrightarrow \Omega L_{BFf} BM_X
\]
which is a weak equivalence.
Now $BFf\simeq \Sigma f$ since $f$ is a map of connected spaces and $JY\simeq\Omega\Sigma Y$ if $Y$ is connected. Moreover, if $X_0$ denotes the basepoint component of $X$, then
\[
\Omega L_{\Sigma f} BM_X\simeq \Omega L_{\Sigma f} X_0\simeq \Omega L_{\Sigma f} X,
\]
and this yields (i).
Finally, it follows from Corollary~\ref{coCACTA} that
\[
L_f \Omega X \simeq L_{Jf}\Omega X\simeq L_{\Omega\Sigma f}\Omega X,
\]
as claimed in part (ii).
\end{proof}

By induction we also have $ L_f\Omega^n X\simeq \Omega^nL_{\Sigma^n f}X $
for all $X$ and~$n\geq0$. This formula also holds for $n=\infty$, by the following argument, which is similar to the preceding one.

Let $\Com$ be the commutative operad, for which $E(n)$ is a point for all~$n$, and let $\psi\colon E_{\infty}\to\Com$ be a cofibrant resolution. The algebras over $\Com$ in pointed simplicial sets are the commutative monoids and the associated reduced monad is the infinite symmetric product $SP^{\infty}$. Connected commutative monoids are GEMs.

For a space $X$, the quotient $X^n/\Sigma_n$ by the symmetric group action does not have the same homotopy type for $n\ge 2$ as $C(n)\times_{\Sigma_n}X^n$, where $C(n)$ is a contractible space with a free $\Sigma_n$-action. For this reason, $E_{\infty}$-spaces are not homotopy equivalent to commutative monoids, in general.
Instead, if we denote by $B^{\infty}X$ the $\Omega$-spectrum associated with a given $E_{\infty}$-space $X$, then there is an adjunction
\[
B^{\infty}:\Ho(E_{\infty}\text{-alg})\rightleftarrows\Ho(\Spectra):\Omega^{\infty},
\]
which restricts, as another instance of \eqref{Leinster}, to an equivalence of categories between the~full subcategory of grouplike $E_{\infty}$-spaces (i.e., infinite loop spaces) and the full~sub\-category of connective spectra;
see \cite[Pretheorem~2.3.2]{Adams2}.

If $Q$ denotes the reduced monad associated with $E_{\infty}$ on pointed simplicial sets, then May's Approximation Theorem \cite{May} implies that $QX\simeq \Omega^{\infty}\Sigma^{\infty}X$ if $X$ is connected.

\begin{corollary}
\label{loopinfinity}
If $f$ is any basepoint-preserving map between connected simplicial sets, then
\begin{itemize}
\item[{\rm (i)}]
$L_f\Omega^{\infty} X\simeq \Omega^{\infty} L_{\Sigma^{\infty} f}X$, and
\item[{\rm (ii)}]
$L_f\Omega^{\infty} X\simeq L_{\Omega^{\infty}\Sigma^{\infty} f}\;\Omega^{\infty} X$
\end{itemize}
for every spectrum $X$.
\end{corollary}

\begin{proof}
Let 
$
F:\Ssets_*\rightleftarrows E_{\infty}\text{-alg}: U
$
be the Eilenberg--Moore factorization of the reduced monad $Q$ associated with~$E_{\infty}$. The functor $L_f$ lifts to $\Ho(E_{\infty}\text{-alg})$ by Theorem~\ref{maybe}, and
it follows from Theorem~\ref{fumfumfum} that $L_fU M\simeq U L_{F f}M$ for every $E_{\infty}$-algebra~$M$.

For a spectrum $X$, we may view $\Omega^{\infty} X$ as an $E_{\infty}$-algebra. Hence,
\[
L_f U\Omega^{\infty} X \simeq U L_{Ff} \Omega^{\infty} X.
\]
Since homotopical localizations preserve $\pi_0$, the functor $L_{Ff}$ restricts to the full subcategory of $\Ho(E_{\infty}\text{-alg})$ whose objects are grouplike, and, since $B^{\infty}FY\simeq B^{\infty}\Omega^{\infty}\Sigma^{\infty} Y\simeq \Sigma^{\infty} Y$ if $Y$ is connected, the functor $L_{B^{\infty}Ff}$ restricts to the full subcategory of connective spectra as explained in Example~\ref{connective}. Hence Theorem~\ref{alfabeta} yields a comparison map of $\beta$ type
\[
L_{Ff} \Omega^{\infty} X\longrightarrow \Omega^{\infty} L_{B^{\infty}Ff} B^{\infty}\Omega^{\infty} X
\]
which is a weak equivalence.
Here $B^{\infty}Ff\simeq \Sigma^{\infty} f$ and, if $X^c$ is the connective cover of $X$, 
\[
\Omega^{\infty} L_{\Sigma^{\infty} f} B^{\infty}\Omega^{\infty} X\simeq \Omega^{\infty} L_{\Sigma^{\infty} f} X^c\simeq \Omega^{\infty} L_{\Sigma^{\infty} f} X,
\]
where the last step uses \eqref{loccoloc}. This yields (i), and it follows from Corollary~\ref{coCACTA} that
\[
L_f \Omega^{\infty} X \simeq L_{Qf}\Omega^{\infty} X\simeq L_{\Omega^{\infty}\Sigma^{\infty} f}\Omega^{\infty} X,
\]
as claimed in part (ii).
\end{proof}

Similarly, for cellularizations we have the following. 

\begin{corollary}
\label{totdual}
If $A$ is any pointed connected simplicial set, then
\begin{itemize}
\item[{\rm (i)}]
$C_A\Omega X\simeq \Omega C_{\Sigma A}X$, and
\item[{\rm (ii)}]
$C_A\Omega X\simeq C_{\Omega\Sigma A}\Omega X$
\end{itemize}
for all pointed simplicial sets $X$, and
\begin{itemize}
\item[{\rm (iii)}]
$C_A\Omega^{\infty} Y\simeq \Omega^{\infty} C_{\Sigma^{\infty} A}Y$, and
\item[{\rm (iv)}]
$C_A\Omega^{\infty} Y\simeq C_{\Omega^{\infty}\Sigma^{\infty} A}\;\Omega^{\infty} Y$
\end{itemize}
for every spectrum $Y$.
\end{corollary}

\begin{proof}
The proof follows the same steps as the proofs of Corollary~\ref{loops} and Corollary~\ref{loopinfinity}, using part~(b) of Theorem~\ref{maybe}.
\end{proof}

Some of the formulas obtained in the preceding results are contained in~\cite{Bo94,Bo96,DF}.
The preservation of loop spaces (and, more generally, spaces with an action of an algebraic theory) by homotopical localizations was also addressed by Badzioch in~\cite{Badzioch}.

\section{Algebras up to homotopy}
\label{homotopyT}

Suppose given a monad $T$ on a model category $\M$. In this section we will not consider model structures on the category $\M^T$ of $T$-algebras, but we will assume that $T$ preserves weak equivalences and hence descends to a monad on the homotopy category~$\Ho(\M)$. As in Section~\ref{homotopicalstructures}, objects of the Eilenberg--Moore category $\Ho(\M)^T$ will be called \emph{$T$-algebras up to homotopy}.

Our aim in this section is to prove that the equivalence $L_fX\simeq L_{Tf}X$
obtained in Corollary~\ref{coCACTA} also holds when $X$ underlies a $T$-algebra up to homotopy, provided that $T$ and $f$ interact in a suitable way, and similarly for cellularizations.	

For this, it will be convenient to work with the Dwyer--Kan construction of homotopy function complexes as simplicial sets of morphisms in the \emph{hammock localization} of $\M$, as in \cite{DK3,DK1} or in the more elaborate version discussed in \cite[\S\,35.6]{DHKS}.
Thus, in this section we choose $\map_{\M}(X,Y)$ to be the colimit of the nerves $N\LH_n(X,Y)$, where $\LH_n(X,Y)$ is the category whose objects are strings of $n$ maps in $\M$ in arbitrary directions
\begin{equation}
\label{string}
X=X_0\longleftrightarrow X_1\longleftrightarrow X_2\longleftrightarrow\cdots
\longleftrightarrow X_{n-1}\longleftrightarrow X_n=Y,
\end{equation}
where the arrows pointing backwards are weak equivalences. A~morphism in $\LH_n (X,Y)$ is a
commuting diagram between strings of the same type. The colimit of nerves is taken along the maps induced by the functors $\LH_n (X,Y)\to \LH_{n+1} (X,Y)$ consisting of adding ${\rm id}_Y$ at the end of~\eqref{string}. Choosing instead to add ${\rm id}_X$ or interpolating identities at any other place would replace $\map_{\M}(X,Y)$ by a weakly equivalent simplicial set, since natural transformations of functors yield simplicial homotopies between maps after taking nerves.

\begin{lemma}
\label{hammocks}
Let $\M$ be a model category and $f\colon A\to B$ any map in $\M$. Let $T\colon\M\to\M$ be a functor that preserves weak equivalences and is equipped with a natural transformation $\eta\colon{\rm Id}\to T$. Suppose given a $Tf$-local object $X$ together with a map $a\colon TX\to X$ such that $a\circ\eta_X\simeq{\rm id}_X$. Then $X$ is $f$-local.
\end{lemma}

\begin{proof}
By assumption, $\map_{\M}(Tf,X)$ is a weak equivalence. Hence, we will achieve our goal if we prove that $\map_{\M}(f,X)$ is a homotopy retract of $\map_{\M}(Tf,X)$ and therefore it is a weak equivalence as well. For this, consider the diagram
\begin{equation}
\label{maps}
\xymatrix{\map_{\M}(B,X) \ar[r] \ar[d] & \map_{\M}(TB,TX) \ar[r] \ar[d] & \map_{\M}(TB,X) \ar[r] \ar[d] & \map_{\M}(B,X) \ar[d] \\
\map_{\M}(A,X) \ar[r] & \map_{\M}(TA,TX) \ar[r] & \map_{\M}(TA,X) \ar[r] & \map_{\M}(A,X)}
\end{equation}
obtained by passing to nerves, for every $n$, the diagram of functors
\begin{equation}
\label{nerves}
\xymatrix{\LH_n(B,X) \ar[r]^-{\bar T} \ar[d]^{f^*} & \LH_n(TB,TX) \ar[r]^-{a_*} \ar[d]^{(Tf)^*} & \LH_{n+1}(TB,X) \ar[r]^-{(\eta_B)^*} \ar[d]^{(Tf)^*} & \LH_{n+2}(B,X) \ar[d]^{f^*} \\
\LH_{n+1}(A,X) \ar[r]^-{\bar T} & \LH_{n+1}(TA,TX) \ar[r]^-{a_*} & \LH_{n+2}(TA,X) \ar[r]^-{(\eta_A)^*} & \LH_{n+3}(A,X)\rlap{,}}
\end{equation}
where $\bar T$ sends each string of maps from $B$ to $X$ to the string obtained by applying $T$ termwise, while $(-)_*$ indicates composition on the right and $(-)^*$ denotes composition on the left.
In~\eqref{nerves} the left-hand square and the middle square commute, while the right-hand square commutes up to a natural transformation after passing to $\LH_{n+4}(A,X)$, in view of
\[
\xymatrix{A \ar[r]^-{{\rm id}} \ar[d]^{{\rm id}} & A\ar[r]^-{f} \ar[d]^{\eta_A}
& B \ar[r]^-{\eta_B} \ar[d]^{\eta_B} & TB \ar[d]^{{\rm id}} \\
A \ar[r]^-{\eta_A} & TA \ar[r]^-{Tf} & TB \ar[r]^-{{\rm id}} & TB\rlap{.}}
\]
Hence, the whole diagram \eqref{maps} commutes up to homotopy, and there only remains to show that the composite of each of its rows is homotopic to the identity.

For this, note first that $(\eta_B)^*\circ a_*=a_*\circ (\eta_B)^*$, and the same happens with~$A$. Moreover, the commutativity of
\begin{equation}
\label{nat}
\xymatrix{B \ar[r]^{\eta_B} & TB \ar@{<->}[r]^-{T\sigma_1} & TE_1 \ar@{<->}[r] & \cdots \ar@{<->}[r] & TE_{n-1} \ar@{<->}[r]^-{T\sigma_n} & TX \ar[r]^{\rm id} & TX \\
B \ar[r]^{\rm id} \ar[u]_{\rm id} & B \ar@{<->}[r]^-{\sigma_1} \ar[u]_{\eta_B} & E_1 \ar@{<->}[r] \ar[u]_{\eta_{E_1}} & \cdots
\ar@{<->}[r] & E_{n-1} \ar@{<->}[r]^-{\sigma_n} \ar[u]_{\eta_{E_{n-1}}} & X \ar[r]^{\eta_X} \ar[u]_{\eta_X} & TX \ar[u]_{\rm id}}
\end{equation}
for every string of maps from $B$ to $X$ yields a natural transformation $(\eta_X)_*\to (\eta_B)^*\circ \bar T$ as in \cite[Theorem~3.2]{Oriol}. Consequently, the composite of the top row in \eqref{maps} is homotopic to the map obtained by passing to nerves the composite
\[
\xymatrix@C+0.5cm{\LH_n(B,X) \ar[r]^-{(\eta_X)_*} & \LH_{n+1}(B,TX) \ar[r]^-{a_*} & \LH_{n+2}(B,X).}
\]
Since $a\circ\eta_X\simeq{\rm id}_X$, the latter map is indeed homotopic to the identity,
as shown in detail in \cite[Theorem~3.1]{Oriol}. The same argument is repeated with $A$ instead of~$B$ to finish the proof that $\map_{\M}(f,X)$ is a homotopy retract of $\map_{\M}(Tf,X)$.
\end{proof}

\begin{theorem}
\label{magic}
Let $\M$ be a model category and let $T$ be a monad on $\M$ preserving weak equivalences and cofibrant objects.
Let $f$ be a map in $\M$ between cofibrant objects such that $L_f$ and $L_{Tf}$ exist. If $T$ preserves $f$\nobreakdash-equivalences between cofibrant objects and $Tf$\nobreakdash-equivalences between cofibrant objects, then
\[
L_fX\simeq L_{Tf}X
\]
whenever $X$ underlies a $T$\nobreakdash-algebra up to homotopy.
\end{theorem}

\begin{proof}
Consider the Eilenberg--Moore factorization of $T$ viewed as a monad on $\Ho(\M)$,
\[
F : \Ho(\M)\rightleftarrows \Ho(\M)^T: U.
\]
Since $T$ preserves $f$\nobreakdash-equivalences between cofibrant objects,
it follows from Theorem~\ref{monad} that there is a localization $L'$ on $\Ho(\M)^T$ such that $U$ preserves and reflects local objects and equivalences, and $L_fUX\simeq UL'X$ naturally for every homotopy $T$\nobreakdash-algebra $X$. Since $T$ also preserves $Tf$\nobreakdash-equivalences, we infer similarly that $L_{Tf}UX\simeq UL''X$ for a localization~$L''$. We next show that $L'\simeq L''$ as in the proof of Proposition~\ref{greattheorem}.

The $L'$\nobreakdash-local objects are those $(X,a)$ in $\Ho(\M)^T$ such that $X$ is $f$\nobreakdash-local, and the $L''$\nobreakdash-local objects are those such that $X$ is $Tf$\nobreakdash-local. Since $T$ preserves $f$\nobreakdash-equivalences between cofibrant objects, $Tf$ is an $f$\nobreakdash-equivalence and hence every $f$\nobreakdash-local object of $\M$ is $Tf$\nobreakdash-local. This tells us that all $L'$\nobreakdash-local objects of $\Ho(\M)^T$ are $L''$\nobreakdash-local.
Conversely, if $(X,a)$ is $L''$-local in $\Ho(\M)^T$, then we may assume that $X$ is fibrant and cofibrant and $Tf$-local. Since $T$ preserves cofibrant objects, we may also assume that $a\colon TX\to X$ is represented by a map in~$\M$. Thus, $a\circ\eta_X\simeq{\rm id}_X$ and, in this situation, Lemma~\ref{hammocks} tells us that $X$ is $f$-local, so $(X,a)$ is $L'$-local.
\end{proof}

As proved in Lemma~\ref{great}, the condition that $T$ preserves $f$-equivalences and $Tf$\nobreakdash-equiv\-alences in Theorem~\ref{magic} is automatically satisfied for every $f$ if the monad $T$ is associated with a simplicial operad. Hence we infer the following general result.

\begin{corollary}
\label{moremagic}
If $T$ is the reduced monad associated with a simplicial operad acting on pointed simplicial sets, then
\[
L_fX\simeq L_{Tf}X
\]
for every map $f$ if $X$ is the underlying space of a $T$\nobreakdash-algebra up to homotopy.
\end{corollary}

\begin{proof}
This is a consequence of Theorem~\ref{magic} and Lemma~\ref{great}.
\end{proof}

\begin{ex}
\label{GEMs}
The infinite symmetric product $SP^{\infty}$ is the reduced monad
associated with the commutative operad.
As shown in \cite[Proposition~1.1]{CRT}, the homotopy algebras over $SP^{\infty}$ coincide with the strict algebras ---this is parallel to the fact that, in stable homotopy, the classes of homotopy $H\ZZ$-module spectra and strict $H\ZZ$\nobreakdash-module spectra coincide, as proved in~\cite{Javier1}.
Corollary~\ref{moremagic} implies that $L_fX\simeq L_{SP^{\infty}f}X$ if $X$ is a GEM, as pointed out in~\cite[Theorem~1.3]{CRT} and generalizing \cite[Corollary~3.2(iii)]{Bo96}. Hence, in particular,
\[
L_fSP^{\infty}X\simeq L_{SP^{\infty}f}\;SP^{\infty}X
\]
for every space $X$ and every map $f$.
\end{ex}

\begin{ex}
\label{JandQ}
The homotopy algebras over the James functor $J$ are the monoids in $\Ho(\Ssets_*)$ with its Cartesian monoidal structure, that is, the homotopy associative $H$\nobreakdash-spaces. 
Since $J$ is the reduced monad associated with the associative operad, we can infer from Corollary~\ref{moremagic} that $L_fX\simeq L_{\Omega\Sigma f}X$ for every homotopy associative $H$-space $X$ and every map $f$ between pointed connected simplicial sets. This is a more general statement than part~(ii) of Corollary~\ref{loops}.

It follows similarly that $L_fX\simeq L_{\Omega^{\infty}\Sigma^{\infty} f}X$
for every map $f$ between pointed connected simplicial sets and every homotopy algebra $X$ over the reduced monad $Q$ associated with an $E_{\infty}$-operad. The homotopy $Q$-algebras are the \emph{$H_{\infty}$\nobreakdash-spaces} in the sense of \cite[\S\,I.3.7]{BMMS}. As shown in~\cite{JN,Noel}, $H_{\infty}$-spaces need not be homotopy equivalent to  $E_{\infty}$-spaces. Thus we have also sharpened part~(ii) of Corollary~\ref{loopinfinity}.
\end{ex}

More generally, for every cofibrant ring spectrum $E$, we consider the monad defined as $TX=\Omega^{\infty}(E\wedge\Sigma^{\infty}X)$ on $\Ho(\Ssets_*)$, and call its algebras \emph{unstable $E$-modules}.
The previous two examples are special cases of the following result, namely $E=H\ZZ$ in Example~\ref{GEMs} and $E=S$ (the sphere spectrum) in the second part of Example~\ref{JandQ}.

\begin{corollary}
\label{plusE}
Let $E$ be a connective cofibrant ring spectrum. For every map $f$ of pointed connected simplicial sets we have
\[
L_fX\simeq L_{\Omega^{\infty}(E\wedge\Sigma^{\infty}f)}X
\]
for all unstable $E$-modules $X$.
\end{corollary}

\begin{proof}
Let $TX=\Omega^{\infty}(E\wedge\Sigma^{\infty}X)$.
Part~(ii) of Proposition~\ref{FGf} applied to the adjunction
\[
\Sigma^{\infty} : \Ho(\Ssets_*) \rightleftarrows \Ho(\Spectra) : \Omega^{\infty}
\]
tells us precisely that $\Sigma^{\infty}$ sends $f$-equivalences to $\Sigma^{\infty}f$-equivalences for every map $f$ of pointed simplicial sets. Since $E$ is assumed to be connective, smashing with $E$ preserves $\Sigma^{\infty}f$-equivalences. This is proved using the derived function spectrum $\DF(-,-)$ as in Section~\ref{modules}. Indeed, if $E$ is connective, then
\[
[E\wedge X,Z]\cong [E,\DF(X,Z)] \cong [E,\DF^c(X,Z)]
\]
so choosing $Z$ to be $\Sigma^{\infty}f$-local yields our claim.

Finally, $\Omega^{\infty}$ sends $\Sigma^{\infty}f$-equivalences to $f$-equivalences by part~(i) of Corollary~\ref{loopinfinity}. Hence $T$ preserves $f$-equivalences for every~$f$, and Theorem~\ref{magic} yields the desired result.
\end{proof}

\begin{theorem}
\label{comagic}
Let $\M$ be a model category and let $T$ be a monad on $\M$ preserving weak equivalences and cofibrant objects. If $A$ is a cofibrant object in $\M$ such that $C_A$ and $C_{TA}$ exist, and $T$ preserves cofibrant $A$\nobreakdash-cellular objects and cofibrant $TA$\nobreakdash-cellular objects, then
\[
C_AX\simeq C_{TA}X
\]
if $X$ underlies a $T$\nobreakdash-algebra up to homotopy.
\end{theorem}

\begin{proof}
The proof follows the same steps as the proof of Theorem~\ref{magic}.
\end{proof}

\begin{ex}
If $A$ and $E$ are cofibrant spectra and $E$ is a homotopy ring spectrum, then, assuming either that $E$ is connective or that $C_A$ commutes with suspension, the monad $TX=E\wedge X$ preserves $A$-cellular objects and $TA$\nobreakdash-cell\-ular objects by the argument given in the proof of Theorem~\ref{mainone}. It then follows from Theorem~\ref{comagic} that
$C_AX\simeq C_{E\wedge A}X$ if $X$ underlies a homotopy left $E$-module. As~a special case, $C_AX\simeq C_{H\ZZ\wedge A}X$ for every $A$ when $X$ is a stable~GEM.
\end{ex}

Using Lemma~\ref{cogreat}, it follows from Theorem~\ref{comagic} that
$C_AX\simeq C_{SP^{\infty}A}X$
for every pointed connected simplicial set $A$ if $X$ is a GEM. Likewise, $C_AX\simeq C_{\Omega\Sigma A}X$ if $X$ is a homotopy associative $H$-space,
and $C_AX\simeq C_{\Omega^{\infty}\Sigma^{\infty}A}X$ if $X$ is an $H_{\infty}$-space.
More generally, the analogue of Corollary~\ref{plusE} for cellularizations reads as follows.

\begin{corollary}
\label{coplusE}
Let $E$ be a connective cofibrant ring spectrum. For every pointed connected simplicial set $A$ we have
\[
C_AX\simeq C_{\Omega^{\infty}(E\wedge\Sigma^{\infty}A)}X
\] 
for all unstable $E$-modules $X$.
\end{corollary}

\begin{proof}
The functor $\Sigma^{\infty}$ sends $A$-cellular pointed simplicial sets to $\Sigma^{\infty}A$-cellular spectra by part~(ii) of Proposition~\ref{coHoFGf}, while smashing with $E$ preserves $\Sigma^{\infty}A$-cellular spectra since $E$ is connective, and $\Omega^{\infty}$ sends those to $A$-cellular spaces by part~(iii) of Corollary~\ref{totdual}. Hence the monad $TX=\Omega^{\infty}(E\wedge\Sigma^{\infty}X)$ preserves $A$-cellular spaces and Theorem~\ref{comagic} applies.
\end{proof}

\end{document}